\documentclass[12pt]{article}
\usepackage[utf8]{inputenc}
\usepackage[english]{babel}
\usepackage[dvips]{graphicx}
\usepackage{indentfirst}
\usepackage{latexsym}
\usepackage{amsthm}
\usepackage{amsfonts}
\usepackage{amssymb}
\usepackage{amsmath}
\usepackage{hyperref}
\theoremstyle{plain}
\newtheorem{theorem}{Theorem}
\newtheorem{lemma}{Lemma}
\newtheorem{corollary}{Corollary}
\newtheorem{proposition}{Proposition}
\newtheorem{assertion}{Claim}
\theoremstyle{definition}

\newtheorem{remark}{Remark}
\newtheorem{example}{Example}
\newtheorem{property}{Property}  
\begin{document}

\noindent UDC 519.17

\title{When a forest, narrowed to an atom of subset algebra, turns out to be a tree} 
\author{V.\,A. Buslov}
\begin{center}
{\bf When a forest, narrowed to an atom of subset algebra, turns out to be a tree} 
\end{center}
\begin{center}
{\large V.\,A. Buslov}
\end{center}

\begin{abstract} 
It is proved that the restriction of a $k$ and $(k-1)$-component directed spanning forest of minimal weight to an atom of the subset algebra generated by the sets of vertices of trees of $k$-component minimal spanning forests is a tree. For spanning minimal forests consisting of fewer components, this property, generally speaking, does not exist.
\end{abstract}
 
This article is based on the work of \cite{V6} and is a continuation of the work of \cite{V7}. Definitions and designations correspond to those adopted in them. We present all the results used in the form of a list of properties, highlighting separately only the fundamental lemma on replacing  arcs \cite[Lemma 1]{V6}.
\section{Notations and definitions}
We denote the set of vertices of a digraph $G$ by ${\cal V}G$, and the set of its arcs by ${\cal A}G$.  
 
The starting object is a weighted digraph $V$, the set of vertices of which is denoted separately for convenience --- ${\cal N}$, $|{\cal N}|=N$;  arcs $(i,j)\in{\cal A}V$ are assigned real weights $v_{ij}$.
We study the spanning subgraphs (subgraphs with a set of vertices ${\cal N}$) of the digraph $V$, which are entering forests. An entering forest (hereinafter simply, forest) is a digraph in which no more than one arc outgoes from each vertex and there are no contours in it. The connected components of a forest are trees. The only vertex of the tree from which the arc does not go is the root. We denote the set of roots of the forest $F$ by ${\cal K}_F$.  

By $T^F_i$ we mean the inclusion-maximal subtree of the forest $F$ rooted at vertex $i$. If $i\in{\cal K}_F$, then $T^F_i$ is a connected component of $F$.

Since only digraphs appear in the work (except for a specially stated point), we use the term graph for them, where this does not lead to misunderstandings.

For a subgraph $G$ of  $V$ and a set ${\cal S}\subseteq {\cal N}$ we define weights   
\begin{equation*}
\Upsilon^G_{\cal S}=\sum_{\begin{smallmatrix}i\in{\cal S} \\ (i,j)\in {\cal A}G\end{smallmatrix}} v_{ij} \ , \ \ \Upsilon^G=\Upsilon^G_{\cal N}=\sum_{(i,j)\in {\cal A}G} v_{ij} \ . 
\end{equation*}
${\cal F}^k$ --- a set of spanning forests consisting of $k$ trees. The minimum weight of $k$-component forests is denoted by 
$$\phi^k=\min_{F\in{\cal
F}^k}\Upsilon^F .$$
If ${\cal F}^k=\emptyset$, we put $\phi^k=\infty$, in particular, $\phi^0=\infty$.

$F\in\tilde{\cal F}^k$ means that $F\in{\cal F}^k$ and $\Upsilon^F=\phi^k$. We will agree to call such forests minimal.

$G|_{\cal S}$ is a subgraph of  $G$ induced by the set ${\cal S}$, that is, ${\cal V}G|_{\cal S}={\cal S }$ and ${\cal A}G|_{\cal S}$ consists of all arcs of  $G$, both ends of which belong to the set ${\cal S}$. $G|_{\cal S}$ is also called the restriction of  $G$ to the set ${\cal S}$.

${\cal F}^k|_{\cal S}$ is a set of subgraphs of $k$-component forests induced by the set ${\cal S}$.

$\tilde{\cal F}^k|_{\cal S}$ is a set of subgraphs of $k$-component forests of minimal weight induced by the set ${\cal S}$.

$G^F_{\uparrow{\cal S}}$ is a graph obtained from $G$ by replacing arcs outgoing from the vertices of ${\cal S}$ with arcs outgoing from the same vertices in $F $.

If there is an arc whose outcome belongs to the set ${\cal S}$, but its entry does not, then we say that the arc outgoes from the set ${\cal S}$. Similarly, if there is an arc whose entry belongs to ${\cal S}$ and whose outcome does not belong, then we say that the arc enters ${\cal S}$. The outgoing neighborhood ${\cal N}^{out}_{\cal S}(G)$ of the set ${\cal S}$ is the set of arc entries outgoing in $G$ from the set ${\cal S }$; The incoming neighborhood ${\cal N}^{in}_{\cal S}(G)$ is defined similarly.

For any subset ${\cal D}\subset {\cal N}$, its complement $\overline{\cal D}={\cal N}\setminus {\cal D}$. 

A family $\mathfrak B $ of nonempty sets ${\cal B}_i$ is called a partition of ${\cal N}$ if ${\cal B}_i\cap{\cal B}_j=\emptyset$ for $i\neq j$, and ${\cal N}=\underset{i}{\cup}{\cal S}_i$.

Let $\mathfrak{B}$ be some family of subsets of ${\cal N}$. The family $\mathfrak A$, consisting of all possible complements and intersections of these subsets, is called an algebra {\it generated} by the family $\mathfrak{B}$.

A non-empty set ${\cal A}$ of algebra $\mathfrak{A}$ is called an {\it atom} if for any element ${\cal B}$ of  algebra $\mathfrak{A}$ either ${\cal A }\cap{\cal B}=\emptyset$, or ${\cal A}\cap{\cal B}={\cal A}$.

\section{Some properties of the arc replacement operation}

The main tool of all the constructions carried out is the operation of replacing arcs outgoing from the vertices of a certain set ${\cal D}$ between two forests $F$ and $G$. Independently of each other, each of two graphs $F_{\uparrow\cal D}^G$ and $G_{\uparrow\cal D}^F$ obtained by this procedure may turn out to be a forest or not. The fundamental statement is the following 

\begin{lemma}\cite[Lemma 1]{V6} 
 Let $F$ and $G$ be forests, ${\cal V}F={\cal V}G={\cal N}$, ${\cal D}\subset{\cal N}$. Then graph $F_{\uparrow\cal D}^G$ is a forest if in $F$ the set ${\cal N}^{in}_{\cal D}(F)$ is unreachable from ${\cal N}^{out}_{\cal D}(G)$.  
\end{lemma} 
\begin{remark} Note that if ${\cal N}^{in}_{\cal D}(F)\cap {\cal N}^{out}_{\cal D}(G)\neq\emptyset$, then these sets are reachable from each other in any forest, since by definition each vertex is considered reachable from itself. In this case, graph $F_{\uparrow\cal D}^G$ may or may not be a forest.
\end{remark}

The simplest special cases of unreachability of ${\cal N}^{in}_{\cal D}(F)$ from ${\cal N}^{out}_{\cal D}(G)$ are situations when in $F$ arcs do not enter ${\cal D}$, or in $G$ arcs do not originate from ${\cal D}$. Then  $F_{\uparrow\cal D}^G$ turns out to be a forest. \cite[Corollaries 1 and 2 from Lemma 1]{V6}.

The situation most in demand for proofs is when both graphs $F_{\uparrow\cal D}^G$ and $G_{\uparrow\cal D}^F$ turn out to be forests.

\begin{property}\cite[Corollaries 3-7 from Lemma 1]{V6}
{\it Let $F$ and $G$ be forests, ${\cal V}F={\cal V}G={\cal N}$, and let $T^F$ and $T^G $ --- some connected components (trees) of forests $F$ and $G$, respectively, ${\cal D}\subset{\cal N}$. Then  graphs $F_{\uparrow\cal D}^G$ and $G_{\uparrow\cal D}^F$ are forests if any of the following points are true:
(a) ${\cal D}={\cal V}T^F$;

(b) ${\cal D}={\cal V}T^F\cap{\cal V}T^G$;

(c) ${\cal D}={\cal V}T^F\setminus{\cal V}T^G$;

(d) ${\cal D}\subset{\cal V}T^F\cap{\cal V}T^G$,  ${\cal N}^{in}_{\cal D}(F)=\emptyset$ and ${\cal N}^{out}_{\cal D}(F)\subset{\cal V}T^F\setminus{\cal V}T^G $;

(e) ${\cal D}\subset{\cal V}T^F\setminus{\cal V}T^G$,  ${\cal N}^{in}_{\cal D}(F)=\emptyset$ and ${\cal N}^{out}_{\cal D}(F)\subset{\cal V}T^G$. }
\end{property}

When constructing minimal forests, the following is useful: 
 \begin{property}\cite[Proposition 1]{V7}
{\it Let $F\in\tilde{\cal F}^n$, $G\in\tilde{\cal F}^m$ and let $m\leq n$ for definiteness. Let also the set ${\cal D}$ be such that graphs $P=F^G_{\uparrow\cal D}$ and $Q=G^F_{\uparrow\cal D}$ are forests, then 

1) if ${\cal D}$ contains the same number of roots of forests $F$ and $G$, then $P\in\tilde{\cal F}^n$ and $Q\in\tilde{\cal F} ^m$;

2) if ${\cal D}$ contains exactly $l=n-m$ more roots of forest $F$ than roots of $G$, then $P\in\tilde{\cal F}^m$ and $Q\in \tilde{\cal F}^n$.}
\end{property}

\section{Subset algebras}

\subsection{Subset algebra generated by a set of forests}

Let ${\cal N}$ be some set of vertices, and let ${\cal F}$ be some set of spanning forests with this set of vertices. Let us introduce the algebra of subsets $\mathfrak{A}_{\cal F}$ of the set ${\cal N}$, which corresponds to the chosen set of forests ${\cal F}$.

We will say that the algebra $\mathfrak{A}_{\cal F}$ is generated by the set of spanning forests ${\cal F}$ if it is generated by the family $\mathfrak{B}_{\cal F}$ consisting of sets ${\cal V}T^F_i$,  $i\in{\cal K}_F$ of tree vertices of forests $F\in{\cal F}$.
Let $\aleph_{\cal F}$ be a family of atoms of this algebra. We call an atom ${\cal A}$ labeled if it contains the root of some forest $F\in{\cal F}$.  We divide the family $\aleph_{\cal F}$ into two disjoint families: the family of labeled atoms $\aleph_{\cal F}^\bullet$ and the family of unlabeled atoms $\aleph_{\cal F}^\circ$.

For any forest $F\in{\cal F}$, the family of sets ${\cal V}T^F_i$,  $i\in{\cal K}_F$ of the vertices of its trees is a partition ${\cal N}$. It follows from this that for a given family of atoms $\aleph_{\cal F}$, the set of forests ${\cal F}$ must have a number of properties. In particular, for any two atoms it is true  

\begin{property}\cite[Lemma 2]{V7}
{\it For any atoms ${\cal A}_1$ and ${\cal A}_2$ of algebra  $\mathfrak{A}_{\cal F}$, there is a forest $F\in{\cal F }$, in which these atoms are contained  in different trees. }
\end{property}

The relative arrangement of three atoms in trees of a forest is described by the following

\begin{figure}[h]
\unitlength=1mm
\begin{center}
\begin{picture}(130,30)

\put(5,20){$\bigcirc$}
\put(18,20){$\bigcirc$}
\put(12,10){$\bigcirc$}
\put(5,24){${\cal A}_1$}
\put(18,24){${\cal A}_2$}
\put(11,5){${\cal A}_3$}
\put(7,23){\oval(10,12)}
\put(20,23){\oval(10,12)}
\put(14,9){\oval(10,12)}
\put(1,0){${\cal F}_*$}

\put(39,20){$\bigcirc$}
\put(54,20){$\bigcirc$}
\put(47,10){$\bigcirc$}
\put(54,15){\oval(17,19)}
\put(41,24){\oval(10,12)}
\put(39,24){${\cal A}_1$}
\put(49,18){${\cal A}_2$}
\put(50,8){${\cal A}_3$}
\put(40,0){${\cal F}_{{\cal A}_1}$}

\put(75,20){$\bigcirc$}
\put(90,20){$\bigcirc$}
\put(82,10){$\bigcirc$}
\put(79,15){\oval(17,18)}
\put(92,24){\oval(10,12)}
\put(90,24){${\cal A}_2$}
\put(79,19){${\cal A}_1$}
\put(77,8){${\cal A}_3$}
\put(73,0){${\cal F}_{{\cal A}_2}$}

\put(110,20){$\bigcirc$}
\put(124,20){$\bigcirc$}
\put(117,10){$\bigcirc$}
\put(119,23){\oval(22,14)}
\put(119,9){\oval(10,12)}
\put(110,24){${\cal A}_1$}
\put(124,24){${\cal A}_2$}
\put(116,5){${\cal A}_3$}
\put(106,0){${\cal F}_{{\cal A}_3}$}

\end{picture} 
\caption{\small Possible arrangement of 3 atoms in  trees of a forest, if they are not all located in one tree at once.}
\label{lem}
\end{center}
\end{figure}
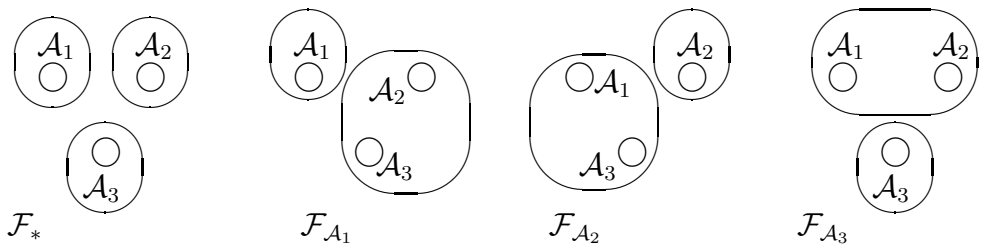

\begin{lemma}
Let ${\cal A}_i\in\aleph_{\cal F}$, $i=1,2,3$, and let the sets ${\cal F}_{{\cal A}_i}\subset{ \cal F}$ consist of forests $F$, in each of which there are two trees such that ${\cal A}_i$ is contained in one of them, and both remaining sets are contained in the other. Let also ${\cal F}_*$ be a set of forests in which all three atoms ${\cal A}_i$ belong to three different trees. Then either ${\cal F}_*\neq \emptyset$ or at most one of the sets ${\cal F}_{{\cal A}_i}$ is empty.
\end{lemma} 

\begin{proof}
Suppose that ${\cal F}$ does not have a forest in which all three atoms of ${\cal A}_i$ belong to its three different trees, and that ${\cal F}$ does not have two types of forests ${ \cal F}_{{\cal A}_i}$ (see Fig. \ref{lem}).  
For definiteness, let these species be ${\cal F}_{{\cal A}_1}$ and ${\cal F}_{{\cal A}_2}$. Then in the set ${\cal F}$ there are no forests in which the atoms ${\cal A}_1$ and ${\cal A}_2$ belong to different trees, which contradicts Property 3. Therefore, either at least one of the sets ${\cal F}_{{\cal A}_1}$ and ${\cal F}_{{\cal A}_2}$ is non-empty, or the set ${\cal F}_*$ is not empty. 
\end{proof}

\subsection{Subset algebras generated by minimal forests}

We will assume that the original graph $V$ is sufficiently dense, in the sense that there is at least one spanning tree, that is, the set ${\cal F}^1$ of spanning forests consisting of one tree is non-empty. Then the sets ${\cal F}^k$, $k\in \{1, 2, \ldots, N\}$ are also non-empty. 
Let us introduce shortened notations: $\mathfrak{B}_k=\mathfrak{B}_{\tilde{\cal F}^k}$, $\mathfrak{A}_k=\mathfrak{A}_{\tilde{\cal F}^k}$, $\aleph_k=\aleph_{\tilde{\cal F}^k}$, $\aleph_k^\circ=\aleph_{\tilde{\cal F}^k}^\circ$, $\aleph_k^\bullet=\aleph_{\tilde{\cal F}^k}^\bullet$.

Fair  
\begin{property}\cite[Theorem 1]{V7}{\it  The sequence of algebras $\mathfrak{A}_k$ is increasing
$$\{{\cal N},\emptyset\}=\mathfrak{A}_1 \subseteq \mathfrak{A}_2
\subseteq \cdots \subseteq \mathfrak{A}_{N-1} \subseteq \mathfrak{A}_N \ =2^{\cal N} ,$$ where $2^{\cal N}$ --- family of all subsets of the 
set ${\cal N}$.}
\end{property}

Not all algebras are distinct and this depends on how many times the system of convex inequalities \cite{V} 

\begin{equation}
\phi^{k-1}-\phi^k\ge\phi^k- \phi^{k+1}   \label{convex}
\end{equation}
for different $k$ there is an equal sign and a strict inequality sign \cite[Theorem 1]{V6}. In particular  \cite[Theorems 2,3, Corollary 1]{V7}

\begin{property}{\it 
Let it be done \begin{equation}
 \phi^{k-1}-\phi^k =\phi^k- \phi^{k+1} \ , \label{equal}
\end{equation}
then
$\mathfrak{A}_k=\mathfrak{A}_{k+1}$; $\aleph_k^\bullet=\aleph_{k+1}^\bullet$ and 
\begin{equation}
\tilde{\cal F}^{k-1}|_{\cal E}\subseteq \tilde{\cal F}^k|_{\cal E}
\supseteq\tilde{\cal F}^{k+1}|_{\cal E},  \ \  {\cal E}\in\aleph_k. 
\label{al}
\end{equation}}
\end{property}

If the system of convexity inequalities for some $k$ contains the strict inequality
\begin{equation}
 \phi^{k-1}-\phi^k >\phi^k- \phi^{k+1} \ , \label{nonequal}
\end{equation}
then a number of properties are fulfilled. 

\begin{property}\cite[Claim 3]{V7}{\it  
 Let (\ref{nonequal}) hold. Then for any ${\cal A}\in\mathfrak{A}_k$, the number of arcs outgoing from the  vertices  of ${\cal A}$ in all forests of $\tilde{\cal F}^k$ is the same. }
\end{property}

\begin{property}\cite[Theorem 4]{V7}{\it  
 Let (\ref{nonequal}) hold, then $|\aleph_k^\bullet|=k$.}
\end{property}

\begin{property}\cite[Claim 4]{V7}{\it   
 Let (\ref{nonequal}) hold, then each  tree of any forest  $F\in\tilde{\cal F}^k$ contains exactly one labeled atom of the algebra $\mathfrak{A}^k$. }
\end{property}

\begin{property}\cite[Claim 6]{V7}{\it    
Let (\ref{nonequal}) hold, $F\in\tilde{\cal F}^k$ and ${\cal E}\in \aleph_k$. Then there exist a forest $H\in\tilde{\cal F }^k$ such that all arcs outgoing from the vertices of the set ${\cal E}$ in the forests $F$ and $H$ coincide, and there are no arcs entering ${\cal E}$ in $H$. }
\end{property}

\begin{property}\cite[Theorem 5]{V7}{\it   
Let (\ref{nonequal}) hold and ${\cal A}\in\mathfrak{A}_k$.  Then the sum of weights of arcs outgoing from vertices of   ${\cal A}$ is the same for all forests of $\tilde{\cal F}^k$.}
\end{property}
This property allows on the algebra of subsets $\mathfrak{A}_k$ when (\ref{nonequal}) is satisfied, introduce the measure $\rho$:

\begin{equation}
\rho_{\cal A}=\Upsilon^F_{\cal A}, \ {\cal A}\in\mathfrak{A}_k , \ F\in\tilde{\cal F}^k. \label{rhoe}
\end{equation}

\begin{property}\cite[Corollary 2 from Theorem 5]{V7}{\it  
 Let (\ref{nonequal}) hold,  and the family of sets ${\cal A}_i\in\mathfrak{A}_k$ forms a partition ${\cal N}$, then $\phi^k=\underset{i} {\sum}\rho_{{\cal A}_i}$. In particular
 $\phi^k=\underset{{\cal E}\in\aleph_k}{\sum} \rho_{\cal E}$.}
\end{property}

\begin{property}\cite[Claim 7]{V7}{\it  
Let (\ref{nonequal}) hold, $F,G\in\tilde{\cal F}^k$, and let ${\cal A}\in\mathfrak{A}_k$ be such that $H=F^G_{\uparrow{\cal A}}$ is a forest. Then $H\in\tilde{\cal F}^k$.}
\end{property}

\begin{property}\cite[Claim 8]{V7}{\it   
Let (\ref{nonequal}) hold, $F\in\tilde{\cal F}^k$, and let ${\cal A}\in\mathfrak{A}_k$ be such that no arcs enters ${\cal A}$ in $F$. Then for any $G\in\tilde{\cal F}^k$, the graph $H=F^G_{\uparrow{\cal A}}$ belongs to $\tilde{\cal F}^k$. }
\end{property}

\begin{property}\cite[Theorem 6]{V7}{\it   
Let (\ref{nonequal}) hold. Then we have the following statements. 

1) If ${\cal M}\in\aleph_k^\bullet$, and $F$ is a forest having outgoing arcs from all but one vertices of $\cal M$, then $\Upsilon^F_ {\cal M}\ge \rho_{\cal M}$. Moreover, if $\Upsilon^F_{\cal M}= \rho_{\cal M}$, then there exists a forest $H\in\tilde{\cal F}^k$ such that $H|_{\cal M}=F|_{\cal M}$;

2) If ${\cal U}\in\aleph_k^\circ$ and  
$F$ is a forest having outgoing arcs from all vertices of ${\cal U}$, then $\Upsilon^F_{\cal U}\ge \rho_{\cal U}$. 
Moreover, if $\Upsilon^F_{\cal U}= \rho_{\cal U}$, then there exists a forest $H\in\tilde{\cal F}^k$ such that all arcs outgoing from vertices of ${\cal U}$ in  $F$ and $H$ coincide.}
\end{property}

\begin{property}\cite[Theorem 7]{V7}{\it    
Let (\ref{nonequal}) hold, $ F\in\tilde{\cal F}^k$, and ${\cal M}\in \aleph_k^\bullet $. Then the  
induced subgraph $F|_{\cal M}$ is a tree.}
\end{property}

\begin{property}\cite[Theorem 8]{V7}{\it  
Let (\ref{nonequal}) hold, and $ F\in\tilde{\cal
F}^{k-1}$. Then there exist a forest $P\in\tilde{\cal F}^{k}$ and an atom ${\cal M}\in\aleph_k^\bullet$ such that all arcs outgoing from the vertices of  ${\cal N}\setminus{\cal M}$ in the forests $F$ and $P$ coincide, and $F|_{\cal M}$ is a tree.}
\end{property}

\begin{property}\cite[Corollary 3 from Theorem 8]{V7}{\it 
Let (\ref{nonequal}) hold. Then for any unlabeled atom 
${\cal U}\in\aleph_k^\circ$}   
\begin{equation}
\tilde{\cal F}^{k-1}|_{\cal U}\subseteq \tilde{\cal F}^k|_{\cal U},  \ \  {\cal U}\in\aleph_k^\circ. 
\label{uat}
\end{equation}
\end{property}

\section{Formulation of the main statement}
 
\begin{figure}[h]
\unitlength=0.9mm
\begin{center}
\begin{picture}(70,54)

\put(0,36){\line(1,0){30}}
\put(23,48){$x$}
\put(23,29){$y_1$}
\put(23,9){$y_2$}

\put(15,41){${\cal X}$}
\put(15,16){${\cal Y}$}
\put(5,51){$\centerdot$}
\put(5,41){$\centerdot$}
\put(5,31){$\centerdot$}
\put(5,21){$\centerdot$}
\put(5,11){$\centerdot$}
\put(5,1){$\centerdot$}

\put(15,6){$\centerdot$}
\put(15,26){$\centerdot$}
\put(15,46){$\centerdot$}

\put(25,6){$\centerdot$}
\put(25,26){$\centerdot$}
\put(25,46){$\centerdot$}

\put(7,2){\vector(2,1){8}}
\put(7,22){\vector(2,1){8}}
\put(7,42){\vector(2,1){8}}

\put(7,11){\vector(2,-1){8}}
\put(7,31){\vector(2,-1){8}}
\put(7,51){\vector(2,-1){8}}

\put(17,6){\vector(1,0){8}}
\put(17,26){\vector(1,0){8}}
\put(17,46){\vector(1,0){8}}

\put(15,27){\oval(30,54)}

\put(64,48){$x$}

\put(45,51){$\centerdot$}
\put(45,41){$\centerdot$}
\put(45,31){$\centerdot$}
\put(45,21){$\centerdot$}
\put(45,11){$\centerdot$}
\put(45,1){$\centerdot$}

\put(55,6){$\centerdot$}
\put(55,26){$\centerdot$}
\put(55,46){$\centerdot$}

\put(65,6){$\centerdot$}
\put(65,26){$\centerdot$}
\put(65,46){$\centerdot$}

\put(47,2){\vector(2,1){8}}
\put(47,22){\vector(2,1){8}}
\put(47,42){\vector(2,1){9}}

\put(47,11){\vector(2,-1){8}}
\put(47,31){\vector(2,-1){8}}
\put(47,51){\vector(2,-1){8}}

\put(57,6){\vector(1,0){8}}
\put(57,26){\vector(1,0){8}}
\put(57,46){\vector(1,0){8}}

\put(65,7){\vector(-1,2){9}}
\put(65,27){\vector(-1,2){9}}
\put(55,27){\oval(30,54)}

\end{picture} 
\caption{\small Subgraph $F|_{\cal U}$ induced by an atom ${\cal U}\in\mathfrak{A}_k$. 
${\cal X}$ is a set of vertices of a connected component including the root $x$, ${\cal Y}={\cal U}\setminus {\cal X}$. On the left -- the hypothesis is unfair. On the right -- the hypothesis is true, $x$ is the only root of the graph $F|_{\cal U}$, ${\cal X}={\cal U}$.  } 
\label{atom}
\end{center}
\end{figure}
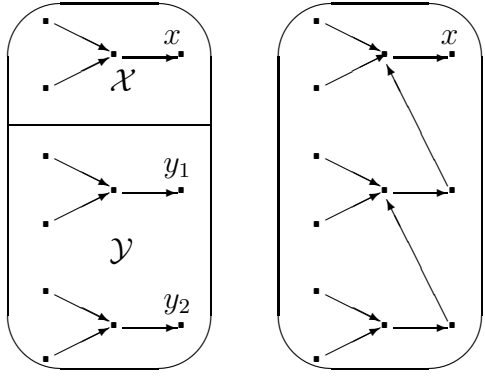

In \cite{V7} it was proven that under (\ref{nonequal}) the forest $F\in\tilde{\cal F}^k$ restricted to a labeled atom of the algebra $\mathfrak{A}_k$ is  a tree (see Property 15). It was also suggested that this property is valid for any atoms, regardless of the sign of the strict or non-strict inequality in (\ref{convex}).   

It will be convenient for us to formulate as a hypothesis the main statement that will be proven.

{\bf Hypothesis.}{\it \ Any forest $F\in\tilde{\cal F}^k$ restricted to an atom of algebra $\mathfrak{A}_k$ is a tree.} 

In Fig.\ref{atom} on the right is a situation corresponding to the hypothesis. 

Let us first make sure that the hypothesis itself needs to be tested only for those $k$ for which the strict inequality holds (\ref{nonequal}). 
 
\section{The case of equality in the system of convexity inequalities}

\begin{proposition}
Let, for some $m$ and $n$, $0<m<n<N$,

\begin{equation}
\phi^{m-1}-\phi^{m}=\phi^{m}-\phi^{m+1}= \ldots= \phi^{n-2}-\phi^{n-1} = \phi^{n-1}- \phi^{n} > \phi^{n}-\phi^{n+1} , 
\label{=>}
\end{equation}
and for index $k=n$ the hypothesis is valid. Then it is valid for any $k=m, m+1, \ldots, n-1$.  
\end{proposition}
\begin{proof}
First let's analyze the right side (\ref{=>}) 
\begin{equation*}
\phi^{n-1}- \phi^{n} > \phi^{n}-\phi^{n+1} \  ,
\end{equation*}
which means that for $k=n$ (\ref{nonequal}) is satisfied. From Properties 16 and 17 (\ref{uat}) it follows that if for any $F\in\tilde{\cal F}^n$ its restriction to an atom of the algebra $\mathfrak{A}_n$ is a tree, then the restriction of any forest from $\tilde{\cal F}^{n-1}$ to an atom of the same algebra is also a tree. The rightmost of the equalities (\ref{=>})
 \begin{equation*}
\phi^{n-2}- \phi^{n-1} = \phi^{n-1}-\phi^{n} \  
\end{equation*}
according to Property 5, it means, in particular, that $\mathfrak{A}_{n-1}=\mathfrak{A}_n$. Thus, the restriction of any forest from $\tilde{\cal F}^{n-1}$ to an atom of its own algebra $\mathfrak{A}_{n-1}$ is a tree. The equalities themselves from (\ref{=>}), according to the same Property 5, determine a wider coincidence of algebras: $\mathfrak{A}_m= \mathfrak{A}_{m+1}=\ldots = \mathfrak{A}_n$. Moreover, from (\ref{al}) it follows that the sets of forests induced by the atom of the algebra $\mathfrak{A}_n$ also coincide: 
\begin{equation*}
\tilde{\cal F}^{m}|_{\cal E}= \tilde{\cal F}^{m+1}|_{\cal E}=\ldots = \tilde{\cal F}^{k-1}|_{\cal E},  \ \  {\cal E}\in\aleph_m=\aleph_{m+1}=\ldots =\aleph_n. 
\end{equation*} 
\end{proof} 

\begin{remark} Note that from the leftmost equality in (\ref{=>}) according to (\ref{al}) it also holds $\tilde{\cal F}^{m-1}|_{\cal E}\subseteq \tilde{\cal F}^k|_{\cal E}$,  ${\cal E}\in\aleph_m$. This means, in particular, that for any forest $F$ from $\tilde{\cal F}^{m-1}$ its restriction to an atom of the algebra $\mathfrak{A}_m$ is a tree when fulfilled (\ref {equal}). However, this circumstance itself does not indicate anything about the properties of the restriction of the forest $F$ to the atoms of the corresponding to it algebra $\mathfrak{A}_{m-1}$, which the system (\ref{=>}) does not allow us to judge. 
\end{remark}

\begin{theorem}
If a hypothesis is valid for all $k$ for which (\ref{nonequal}) holds, then it is valid for all $k$.
\end{theorem}
\begin{proof}
Taking into account the proposition proved above, it remains to consider only the situation when a strict inequality in the convexity system occurs for the first time. Exactly, let it be done
\begin{equation*}
\phi^{n-1}-\phi^n > \phi^n- \phi^{n+1} = \phi^{n+1}-\phi^{n+2} = \ldots = \phi^{N-1}-\phi^N \ , 
\end{equation*}
and for index $k=n$ the hypothesis is valid. 
Due to Property 5 $\mathfrak{A}_{n+1}=\mathfrak{A}_{n+2}=\ldots = \mathfrak{A}_N$. The algebra $\mathfrak{A}_N$ is a Boolean and all its atoms are singletons. A forest induced by a singleton subset is an empty tree.  
\end{proof}

\section{Strict inequality in the system of convexity inequalities}

Let now (\ref{nonequal}) be satisfied. 
First of all, we note that for $k=1$ the strict inequality
\begin{equation*}
\phi^0-\phi^1>\phi^1-\phi^2
\end{equation*}
performed automatically, since $\phi^0=\infty$ due to the absence of forests without roots. The algebra $\mathfrak{A}_1$ is trivial and its only atom is the entire set of vertices ${\cal N}$. The spanning tree, induced by the set  of its own vertices, does not change at all and remains a tree. So for $k=1$ the hypothesis is also satisfied automatically. 

Since in a situation of strictly inequality (\ref{nonequal}) by Property 15 the hypothesis is valid for labeled atoms, then only unlabeled atoms are subject to consideration.

\subsection{Restrictions on arcs outgoing an atom}

Further, everywhere in this section where the hypothesis is assumed to be unfulfilled, the same notations are used. In order not to repeat them every time in the proofs, let us agree that the unlabeled atom ${\cal U}$ belongs to the tree $T^F$ of some forest $F\in\tilde{\cal F}^k$. Several arcs outgo from the atom ${\cal U}$. It is enough to highlight two of them. Let them come from vertices $x$ and $y$. ${\cal X}$ and ${\cal Y}$ are sets of vertices of connected components of the induced subgraph $F|_{\cal U}$ whose roots are the vertices $x$ and $y$, respectively. We also assume that the arc outgoing from $x$ in $F$ enters the atom ${\cal E}$.   
 
The following lemma defines the main restrictions on arcs outgoing from an atom.

 \begin{lemma}
Let (\ref{nonequal}), $F\in\tilde{\cal F}^k$ and ${\cal U}\in \aleph_k^\circ$ be satisfied. In any forest $G\in\tilde{\cal F}^k$, at most one of the entry points of arcs outgoing from ${\cal U}$ in  $F$ is not in the same tree with ${\cal U}$.
\end{lemma}
\begin{proof}

By Property 9, without loss of generality, we can assume that in  $F$ no arcs enter atom ${\cal U}$. Let's prove by contradiction. Suppose that in $F$ several arcs outgo from ${\cal U}$ (two are enough) and there exists a forest $G\in\tilde{\cal F}^k$ in which both arcs originating in $F$ from the vertices $x$ and $y$ of the atom ${\cal U}$ are not in the same tree as ${\cal U}$. In $G$ let the atom ${\cal E}$ (where the arc outgoing from $x$ in $F$ ends) be located in the tree $T^G$. Possible situations are shown in Fig. \ref{pair}. We have  ${\cal X}\subset{\cal V}T^F\setminus{\cal V}T^G$, ${\cal N}^{in}_{\cal X}(F)=\emptyset$, ${\cal N}^{out}_{\cal X}(F)\subset{\cal V}T^G$. By Property 1(e)  graphs $R=G^F_{\uparrow\cal X}$ and $F^G_{\uparrow\cal X}$ are forests. Since the atom ${\cal U}$ is not labeled, it does not contain the roots of any forest from $\tilde{\cal F}^k$. Therefore, in both $F$ and $G$, arcs outgo from all vertices of ${\cal U}$. This means that all vertices of ${\cal X}$ also have arcs outgoing at both $F$ and $G$. Then, by Property 2, $R\in\tilde{\cal F}^k$. By construction in  $R$, the set ${\cal E}$ and the vertex $x$ (and the set ${\cal X}$) are in the same tree. This means that atom ${\cal U}$ is also in the same tree.
  
Note that in  $G$ there are no arcs outgoing from ${\cal V}T^F\setminus{\cal V}T^G$ and entering ${\cal V}T^G$. Thus, in  $R$, the arc outgoing from $x$ is the only one that has an outcome in ${\cal V}T^F\setminus{\cal V}T^G$ and has an end in ${ \cal V}T^G$. This means that the tree of $R$ containing the set ${\cal V}T^G$ and the vertex $x$ (and the set ${\cal X}$) contains the entire set ${\cal U}$ only in the case, if in $G$ the set ${\cal X}$ is reachable from any vertex of the set ${\cal Y}$. By replacing arcs between forests $F$ and $G$, but on the set ${\cal Y}$, we similarly verify that in  $G$ the set ${\cal Y}$ is reachable from any vertex of the set ${\cal X}$.  But then $G$ inevitably contains a contour including some vertices from ${\cal X}$ and some vertices from ${\cal Y}$. This contradicts the fact that $G$ is a forest.
\end{proof}

\begin{figure}[h]
\unitlength=1mm
\begin{center}
\begin{picture}(110,33)

\put(2,29){a)}
\put(4,18){${\cal X}$}
\put(9,18){$x$}
\put(12,19){$\bullet$}
\put(4,12){${\cal Y}$}
\put(9,11){$y$}
\put(12,11){$\bullet$}
\put(9,16){\oval(14,14)}
\put(14,14){\oval(26,26)}
\put(26,22){\oval(20,12)}
\put(26,10){\oval(20,12)}
\put(22,21){\oval(7,7)}
\put(20,20){${\cal E}$}
\put(22,11){\oval(7,7)}
\put(21,10){${\cal S}$}
\put(10,3){$T^F$}
\put(28,20){$T^G$}
\put(28,7){${T'}^G$}
\put(2,16){\line(1,0){14}}
\put(14,20){\vector(1,0){5}}
\put(14,12){\vector(1,0){5}}

\put(40,29){b)}
\put(42,18){${\cal X}$}
\put(47,18){$x$}
\put(50,19){$\bullet$}
\put(42,12){${\cal Y}$}
\put(47,11){$y$}
\put(50,11){$\bullet$}
\put(47,16){\oval(14,14)}
\put(53,13){\oval(28,25)}
\put(61,17){\oval(10,30)}
\put(61,21){\oval(7,7)}
\put(59,20){${\cal E}$}
\put(61,11){\oval(7,7)}
\put(60,10){${\cal S}$}
\put(49,3){$T^F$}
\put(58,26){$T^G$}
\put(40,16){\line(1,0){14}}
\put(52,20){\vector(1,0){6}}
\put(52,12){\vector(1,0){6}} 

\put(77,29){c)}
\put(79,18){${\cal X}$}
\put(84,18){$x$}
\put(87,19){$\bullet$}
\put(79,12){${\cal Y}$}
\put(84,11){$y$}
\put(87,11){$\bullet$}
\put(84,16){\oval(14,14)}
\put(90,13){\oval(30,25)}
\put(98,20){\oval(12,25)}
\put(98,16){\oval(10,10)}
\put(97,14){${\cal E}$}
\put(88,3){$T^F$}
\put(96,26){$T^G$}
\put(77,16){\line(1,0){14}}
\put(89,20){\vector(1,0){7}}
\put(89,12){\vector(1,0){7}}

\end{picture} 
\caption{\small By assumption, in $F$ two arcs outgo from the atom ${\cal U}={\cal X}\cup{\cal Y}$ (these arcs are shown in the figure) and there is a forest in which both ends of these arcs are not in the same tree with ${\cal U}$.}
\label{pair}
\end{center}
\end{figure}
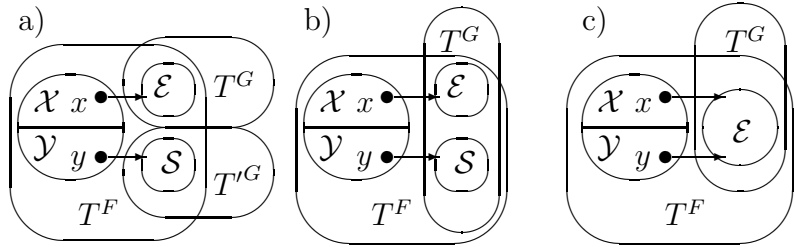

A direct consequence of Lemma 3 is the absence of multiple arcs outgoing from one atom to another.

\begin{corollary} Let (\ref{nonequal}) hold, $ F\in\tilde{\cal F}^k$, ${\cal U}\in \aleph_k^\circ$ and ${\cal E}\in \aleph_k$. Then from ${\cal U}$ at most one arc outgoes ending at ${\cal E}$.
\end{corollary}

\begin{proof}
Let us assume the opposite, namely, that from ${\cal U}$ outgo several arcs ending at ${\cal E}$. Since ${\cal U}$ and ${\cal E}$ are atoms, then by Property 2 there is a forest $G\in\tilde{\cal F}^k$ in which they belong to different trees, as shown in fig. \ref{pair} c). However, according to Lemma 3, such a situation is excluded.
\end{proof}

\begin{theorem}
Let (\ref{nonequal}) hold, $ F\in\tilde{\cal F}^k$, ${\cal U}\in \aleph_k^\circ$, ${\cal M}\in \aleph_k^\bullet$. Then, if in  $F$ there is an arc outgoing from ${\cal U}$ and ending in ${\cal M}$, then this arc is the only one outgoing from ${\cal U}$.
\end{theorem}

\begin{proof}

Let us show by contradiction that an unlabeled atom cannot produce two arcs into different atoms, one of which is labeled. Let in  $F$ an arc outgoes from an unlabeled atom ${\cal U}$ into a labeled atom ${\cal M}$ (its origin is a vertex $y$) and an arc outgoes into an unlabeled atom ${\cal E} $. By Property 9, without loss of generality, we assume that in the forest $F$ no arcs enter the atom ${\cal U}$.

\begin{figure}[h]
\unitlength=1mm
\begin{center}
\begin{picture}(130,30)
\put(5,20){$\bigcirc$}
\put(18,20){$\bigcirc$}
\put(12,10){$\bigodot$}
\put(5,24){${\cal U}$}
\put(19,24){${\cal E}$}
\put(11,5){${\cal M}$}
\put(7,23){\oval(10,12)}
\put(20,23){\oval(10,12)}
\put(14,9){\oval(10,12)}
\put(1,0){$a)$}

\put(39,20){$\bigcirc$}
\put(54,20){$\bigcirc$}
\put(46,10){$\bigodot$}
\put(54,15){\oval(17,19)}
\put(41,24){\oval(10,12)}
\put(40,24){${\cal U}$}
\put(51,19){${\cal E}$}
\put(51,8){${\cal M}$}
\put(40,0){$b)$}

\put(75,20){$\bigcirc$}
\put(90,20){$\bigcirc$}
\put(81,10){$\bigodot$}
\put(79,15){\oval(17,18)}
\put(92,24){\oval(10,12)}
\put(90,24){${\cal E}$}
\put(80,19){${\cal U}$}
\put(75,8){${\cal M}$}
\put(73,0){$c)$}

\put(110,20){$\bigcirc$}
\put(124,20){$\bigcirc$}
\put(117,10){$\bigodot$}
\put(119,23){\oval(22,14)}
\put(119,9){\oval(10,12)}
\put(111,24){${\cal U}$}
\put(125,24){${\cal E}$}
\put(116,5){${\cal M}$}
\put(106,0){$d)$}

\end{picture} 
\caption{\small If we assume that in some forest the atom ${\cal U}$ has arcs ending in ${\cal E}$ and in ${\cal M}$, then by Lemma 3 there are not forests of the types $a)$ and $b) $, and then by Lemma 2 there must exist forests of both type $c)$ and type $d)$.}
\label{lemma}
\end{center}
\end{figure}
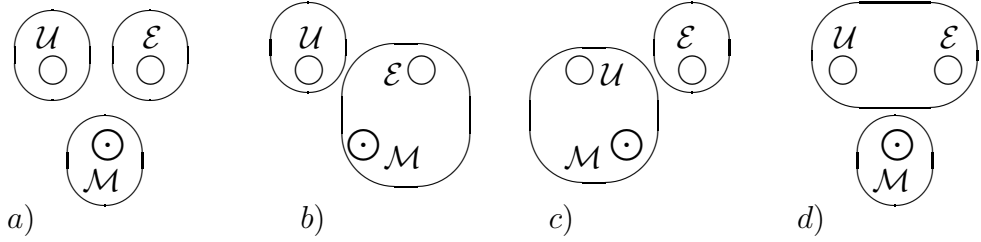

Assumption that there is a forest $G\in\tilde{\cal F}^k$ in which ${\cal U}$ is not in the same tree with both ${\cal E}$ and ${\cal M }$ (see Fig. \ref{lemma} options $a)$ and $b)$), is in contradiction with Lemma 3.
Then, by Lemma 2, there are forests of both type $c)$ and type $d)$ from Fig.\ref{lemma}. Let us make sure that option $c)$ leads to a contradiction. Indeed, let in $G\in\tilde{\cal F}^k$  atoms ${\cal E}$ and ${\cal U}$ be in different trees: for definiteness, in trees $T^G$ and ${T'}^G$ respectively, and ${\cal E}$ does not belong to the tree rooted at ${\cal M}$ (see  Fig.\ref{th4} on the left).
Consider graph $Q=G^F_{\uparrow{\cal U}}$. We have: in $F$ one arc outgoing from ${\cal U}$ enters atom ${\cal M}$, which is labeled, and arcs from it do not outgo in any forest from $\tilde{ \cal F}^k$ because satisfied (\ref{nonequal}). Another arc outgoing from ${\cal U}$ in $F$ enters the tree ${T}^G$. But in $G$ there are no arcs outgoing from the set of vertices of this tree (or any other), nor are there any arcs entering ${T'}^G$. But it is precisely this tree in $G$ that contains  both the atom ${\cal U}$ and the set ${\cal N}^{in}_{\cal U}(G)$ (in $G$ there may be arcs entering the atom ${\cal U}$). Thus, no sequence of arcs from ${\cal N}^{out}_{\cal U}(F)$ leads in $G$ to the set ${\cal N}^{in}_{\cal U}(G)$. Then, by Lemma 1, the graph $Q$ is a forest. By Property 12 $Q\in\tilde{\cal F}^k$. However, by construction in  $Q$, parts of the set ${\cal U}$ are in different trees (see Fig.\ref{th4} on the right), which is impossible, since ${\cal U}$ is an atom.
\end{proof}
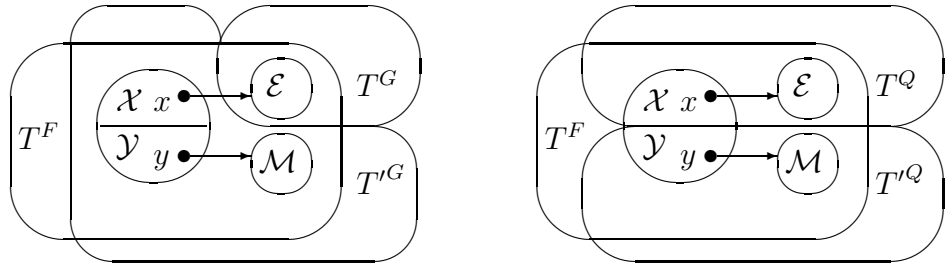
\begin{figure}[h]
\unitlength=1mm
\begin{center}
\begin{picture}(125,33)
\put(14,20){${\cal X}$}
\put(19,20){$x$}
\put(22,21){$\bullet$}
\put(14,14){${\cal Y}$}
\put(19,13){$y$}
\put(22,13){$\bullet$}
\put(19,18){\oval(15,15)}
\put(22,16){\oval(44,26)}
\put(31,6){\oval(46,12)[b]}
\put(41,12){\oval(26,12)[tr]}
\put(18,29){\oval(20,10)[t]}
\put(8,6){\line(0,1){23}}
\put(54,6){\line(0,1){6}}
\put(41,26){\oval(27,16)}
\put(36,23){\oval(8,8)}
\put(34,22){${\cal E}$}
\put(36,13){\oval(8,8)}
\put(33,12){${\cal M}$}
\put(1,15){$T^F$}
\put(46,22){${T}^G$}
\put(46,9){${T'}^G$}
\put(12,18){\line(1,0){14}}
\put(24,22){\vector(1,0){8}}
\put(24,14){\vector(1,0){8}} 

\put(84,20){${\cal X}$}
\put(89,20){$x$}
\put(92,21){$\bullet$}
\put(84,14){${\cal Y}$}
\put(89,13){$y$}
\put(92,13){$\bullet$}
\put(89,18){\oval(15,15)}
\put(92,16){\oval(44,26)}
\put(100,9){\oval(48,18)}
\put(100,26){\oval(48,16)}
\put(106,23){\oval(8,8)}
\put(104,22){${\cal E}$}
\put(106,13){\oval(8,8)}
\put(103,12){${\cal M}$}
\put(71,15){$T^F$}
\put(115,22){${T}^Q$}
\put(115,9){${T'}^Q$}
\put(82,18){\line(1,0){14}}
\put(94,22){\vector(1,0){8}}
\put(94,14){\vector(1,0){8}} 

\end{picture} 
\caption{\small By assumption, two arcs in $F$ outgo from atom ${\cal U}={\cal X}\cup{\cal Y}$ and one of them enters the labeled atom ${\cal M}$. Then in the forest $Q=G^F_{\uparrow{\cal U}}$ (on the right) the sets ${\cal X}$ and ${\cal Y}$ are in different trees.}
\label{th4}
\end{center}
\end{figure}   

\begin{remark}
In the above proof, it is essential that ${\cal M}$ is a labeled atom. Arcs do not come from it under condition  (\ref{nonequal}). Thus, in all forests from $\tilde{\cal F}^k$ not a single atom is reachable from ${\cal M}$. Thanks, in particular, to this circumstance, it is possible to use Lemma 1.
\end{remark}

\begin{lemma}
Let (\ref{nonequal}) hold, ${\cal U},{\cal E},{\cal S} \in \aleph_k^\circ$, and there is a forest $G\in\tilde{\cal F} ^k$ such that in it arcs do not enter ${\cal U}$ and one of the atoms ${\cal E}$ or ${\cal S}$ is not in the same tree as ${\cal U }$. Then in any forest $F\in\tilde{\cal F}^k$ there is at most one arc outgoing from ${\cal U}$ to ${\cal E}\cup {\cal S}$.
\end{lemma}

\begin{proof}
Suppose that in some forest $F\in\tilde{\cal F}^k$ two arcs outgo from the unlabeled atom ${\cal U}$ to the unlabeled atoms ${\cal E}$ and ${\cal S} $. In $F$, an arc outgoing from ${\cal U}$ to ${\cal E}$ has the starting point $x$, and an arc outgoing from ${\cal U}$ to ${\cal S }$, has the starting point at vertex $y$.

\begin{figure}[h]
\unitlength=1mm
\begin{center}
\begin{picture}(125,33)
\put(14,20){${\cal X}$}
\put(19,20){$x$}
\put(22,21){$\bullet$}
\put(14,14){${\cal Y}$}
\put(19,13){$y$}
\put(22,13){$\bullet$}
\put(19,18){\oval(15,15)}
\put(22,16){\oval(44,26)}
\put(31,6){\oval(46,12)[b]}
\put(41,12){\oval(26,12)[tr]}
\put(18,29){\oval(20,10)[t]}
\put(8,6){\line(0,1){23}}
\put(54,6){\line(0,1){6}}
\put(41,26){\oval(27,16)}
\put(36,23){\oval(8,8)}
\put(34,22){${\cal E}$}
\put(36,13){\oval(8,8)}
\put(34,12){${\cal S}$}
\put(1,15){$T^F$}
\put(46,22){${T}^G$}
\put(46,9){${T'}^G$}
\put(12,18){\line(1,0){14}}
\put(24,22){\vector(1,0){8}}
\put(24,14){\vector(1,0){8}} 

\put(84,20){${\cal X}$}
\put(89,20){$x$}
\put(92,21){$\bullet$}
\put(84,14){${\cal Y}$}
\put(89,13){$y$}
\put(92,13){$\bullet$}
\put(89,18){\oval(15,15)}
\put(92,16){\oval(44,26)}
\put(88,6){\oval(20,12)[b]}
\put(111,27){\oval(26,18)[br]}
\put(101,27){\oval(46,12)[t]}
\put(78,6){\line(0,1){23}}
\put(111,9){\oval(26,18)}
\put(106,23){\oval(8,8)}
\put(104,22){${\cal E}$}
\put(106,13){\oval(8,8)}
\put(104,12){${\cal S}$}
\put(71,15){$T^F$}
\put(115,22){${T}^H$}
\put(115,9){${T'}^H$}
\put(82,18){\line(1,0){14}}
\put(94,22){\vector(1,0){8}}
\put(94,14){\vector(1,0){8}} 

\end{picture} 
\caption{\small By assumption, two arcs outgo in $F$ from the atom ${\cal U}={\cal X}\cup{\cal Y}$. Situations are depicted for forests in which one of the entries of these arcs is in the same tree as ${\cal U}$, and the other is not.}
\label{hip}
\end{center}
\end{figure}
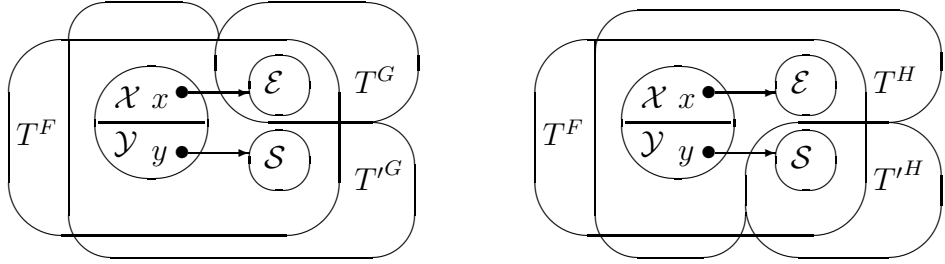

Due to the symmetry of the situation with respect to the atoms ${\cal E}$ and ${\cal S}$, we assume for definiteness that the conditions of the lemma are satisfied by the forest $G$ shown in Fig. \ref{hip} on the left. 
 Since in $G$ arcs do not enter ${\cal U}$, then by Property 13 the forest $Q=G^F_{\uparrow{\cal U}}$ belongs to the set $\tilde{\cal F}^ k$. In $Q$, no arcs enter the set ${\cal X}$, since in $F$ there were no arcs with an origin in ${\cal Y}$ and an entry in ${\cal X}$, and  $G$ has no arcs going into ${\cal U}$.  But then, by construction in $Q$, the sets ${\cal X}$ and ${\cal Y}$ are in different trees (see Fig. \ref{th4} on the right with ${\cal M}$ replaced by ${\cal S}$), which contradicts the fact that the set ${\cal U}={\cal X}\cup{\cal Y}$ is an atom.
 \end{proof}

To fully justify the hypothesis, it is enough to prove that forest $G$ required in Lemma 4 certainly exists.

In the formulation of Lemma 4 it was required that no arcs enter the atom ${\cal U}$ (which allows us to use Property 13). It would seem that Property 9 allows us to get rid of this limitation and fully justify the hypothesis. However, alas, this is not entirely true. In the forest $F$, without loss of generality (by virtue of Property 9), we can assume that arcs do not enter the atom ${\cal U}$ (but this is not required for the proof). But with the forest $G$ the situation is different. It must satisfy an important special condition. Namely, that at least one of the atoms ${\cal E}$ or ${\cal S}$ should not be in the same tree with ${\cal U}$. Therefore, it is, generally speaking, impossible to additionally demand that in $G$ no arcs enter the atom ${\cal U}$. Let us make sure that, nevertheless, both requirements are compatible.

\subsection{Connected components of the restriction of a minimal forest to atoms of algebra $\mathfrak{A}_k$}

Let us find out what properties any connected component has (without assuming that it is the only one) of the subgraph $F|_{\cal U}$ generated by  unlabeled atom ${\cal U}$. Since we do not know in advance the structure of  (unlabeled) atoms of algebra $\mathfrak{A}_k$, we are forced to assume that the induced subgraph $F|_{\cal U}$ itself is a forest consisting of a certain number of connected components ( see fig. \ref{atom} on the left). From every root of these components in $F$ outgoes the single arc. The entry of each of these arcs does not belong to the set ${\cal U}$, since if there were a direct connection between the two components, then they would constitute a single component.

\begin{theorem}
Let (\ref{nonequal}) hold, $F\in\tilde{\cal F}^k$, ${\cal U}\in\aleph_k^\circ$ and ${\cal X}$ is the set of vertices of an arbitrary connected component of the induced subgraph $F|_{\cal U}$. Then 

1) there is a forest $H\in\tilde{\cal F}^k$ such that all arcs outgoing from the vertices of the set ${\cal X}$ in the forests $F$ and $H$ coincide, and the set ${\cal X}$ has no entering arcs.

2) for any $G\in\tilde{\cal F}^k$ is satisfied

\begin{equation}
\Upsilon^G_{\cal X}=\Upsilon^F_{\cal X} \ .  
\label{gxfx}
\end{equation}
If in this case graph $F'=F^G_{\uparrow{\cal X}}$ is a forest (in particular, when in  $F$ there are no arcs entering ${\cal X}$), then $ F'\in\tilde{\cal F}^k$.

3) if $G$ is a forest in which arcs outgo from all vertices of the set ${\cal X}$, then $\Upsilon^G_{\cal X}\ge \Upsilon^F_{\cal X}$. 
Moreover, if $\Upsilon^G_{\cal X}= \Upsilon^F_{\cal X} $, then there is a forest $F'\in\tilde{\cal F}^k$ such that all arcs outgoing from the vertices of the set ${\cal X}$ in the forests $G$ and $F'$ coincide. 
\end{theorem}
 
\begin{proof}

1) According to Property 9, there is a forest $H\in\tilde{\cal F}^k$ in which no arcs enter the atom ${\cal U}$, and arcs outgoing from the vertices of ${\cal U}$ are the same , as in $F$. But then there are no arcs in $H$ going into ${\cal X}$. Indeed, by condition, there are no arcs in $F$ with an outcome in ${\cal Y}={\cal U}\setminus {\cal X}$ and an entry in ${\cal X}$.

2) Let $G\in\tilde{\cal F}^k$. To check equality (\ref{gxfx}), without loss of generality, we can assume from Property 9 that in $F$ no arcs enter the atom ${\cal U}$. Note that 
then it has no incoming arcs not only in ${\cal X}$, but also in ${\cal Y}$. Then, by Lemma 1,  graphs $F'=F^G_{\uparrow{\cal X}}$ and $F''=F^G_{\uparrow{\cal Y}} $ are forests. And they belong to ${\cal F}^k$, since in both $F$ and $G$ arcs outgo from all vertices of the unlabeled atom ${\cal U}$.  Therefore, weights of these forests are not less than the weight of $F$. Forests $F'$ and $F''$ differ from  $F$ only in arcs outgoing from the vertices of the  sets ${\cal X}$ and ${\cal Y}$, respectively. Means   

\begin{equation}
\Upsilon^{F}_{\cal X}\leq \Upsilon^{F'}_{\cal X}, \ \  \Upsilon^{F}_{\cal Y}\leq \Upsilon^{F''}_{\cal Y} \ .
\label{xy}
\end{equation}
By construction, $\Upsilon^{F'}_{\cal X}= \Upsilon^{G}_{\cal X}$ and $\Upsilon^{F''}_{\cal Y}= \Upsilon^ {G}_{\cal Y}$, and by Property 10 $\Upsilon^{G}_{\cal U}= \Upsilon^F_{\cal U}$. So  

\begin{equation*}
\Upsilon^{F'}_{\cal X}+\Upsilon^{F''}_{\cal Y} =\Upsilon^{G}_{\cal X}+\Upsilon^{G}_{\cal Y}=\Upsilon^{G}_{\cal U}= \Upsilon^F_{\cal U}=\Upsilon^{F}_{\cal X}+\Upsilon^{F}_{\cal Y} \ , 
\end{equation*}
from which, taking (\ref{xy}) into account, it follows (\ref{gxfx}). That is, both forests $F'$ and $F''$ are minimal. 

3) Let $G$ be a forest in which arcs outgo from all vertices of the set ${\cal X}$.  Then, similarly to the previous point, for the forest $F'=F^G_{\uparrow{\cal X}}$ we have $F'\in {\cal F}^k$ and $\Upsilon^{F}_{\cal X}\leq \Upsilon^{F'}_{\cal X}=\Upsilon^{G}_{\cal X}$. Let now $\Upsilon^G_{\cal X}= \Upsilon^F_{\cal X} $. Thus, $\Upsilon^{F'}_{\cal X}= \Upsilon^F_{\cal X}$. Since forests $F'$ and $F$ differ only in the arcs outgoing from the vertices of the set ${\cal X}$, we conclude that $\Upsilon^F=\Upsilon^{F'}$. Thus, $F'\in\tilde{\cal F}^k$.
\end{proof}

The equalities (\ref{gxfx}) mean that the measure $\rho$ (\ref{rhoe}), minimally defined on the atoms of the algebra $\mathfrak{A}_k$, can be extended to more refined subsets of any unlabeled atom $ {\cal U}$ if the corresponding induced subgraph $F|_{\cal U}$ is disconnected, namely --- on the sets of vertices of its connected components.   This theorem actually reproduces for such sets Properties 9-14 that  elements of the algebra $\mathfrak{A}_k$ and its atoms have (thus, the theorem indirectly indicates that $F|_{\cal U}$ is connected ).  
And the presence of these properties will allow  
construct the forest used in Lemma 4.     

\section{Building a forest with special properties} 

To construct the required forest, the existing properties of the atoms of the algebra $\mathfrak{A}_k$ are not enough. Without further consideration, everything that can be extracted from the fact that the atom is indivisible has already been done. 

The family $\aleph_k$ is a partition of the vertex set ${\cal N}$, but for this partition it is unknown whether the graph $F|_{\cal U}$ is connected for ${\cal U}\in\aleph_k$ ( in fact, the work lies in checking this circumstance). Let's approach the issue from the other side. We will immediately consider a partition $\mathfrak{N}$ of the tree $T$ (or forest $F$) such that $T|_{\cal X}$ is a tree, for ${\cal X}\in \mathfrak{N}$. Having found out the properties of such a partition, we will be able to construct the forest appearing in Lemma 4.

The next two paragraphs are devoted to some properties of trees (forests) regardless of the presence of a weight function.
\subsection{Simple properties of trees (forests)}

First of all, let's note the following. Forest (incoming) is a graph in which at most one arc outgoes from each vertex and there are no contours in the graph. Removing any vertices and/or arcs does not violate these properties. Therefore, any subgraph of a forest (in particular, a tree) is a forest. This means that the graph induced by an arbitrary subset of the set of vertices of a forest (tree) is also a forest.
The connected components of this forest are trees with fewer vertices if there is more than one component.  

Let's formulate a few simple statements about the trees we will use.

\begin{assertion}
Any two subtrees of a tree either do not intersect, or one of them contains the root of the other.
\end{assertion}

\begin{proof}
In any tree there is only one path from a selected vertex to the root of the tree. Let the tree $T$ be rooted at vertex $\delta$, and let its two subtrees $T'$ and $T''$ be rooted at vertices $\alpha$ and $\beta$. And let some vertex $\gamma$ belong to the intersection of these two subtrees. If an arc does not outgo from it in any of the subtrees, then it is the root of both of them. If an arc does not outgo from it in only one of the subtrees, then it is the root of this subtree. If the arcs originate in both trees, then this is the same arc, which in the tree $T$ itself originates from this vertex (the outgoing arc is unique for every vertex except the root). This arc belongs to both subtrees. This means that its entry also belongs to both subtrees. We move to this vertex and consider whether an arc outgoes from it. And so on until we get to the root of one of the subtrees.
\end{proof}

\begin{assertion}
Let $T$ be a tree, $T_\alpha$ and $T_\beta$ be its two inclusion-maximal connected components with roots at $\alpha$ and $\beta$, respectively. Then either ${\cal V}T_\alpha\cap {\cal V}T_\beta= \emptyset$, or one of these components is a subgraph of the other.
\end{assertion}
\begin{proof}
If 
$\alpha\notin{\cal V}T_\beta$ and $\beta\notin{\cal V}T_\alpha$,
 then according to Claim 1 ${\cal V}T_\alpha\cap {\cal V}T_\beta= \emptyset $. Let now, say, $\alpha\in{\cal V}T_\beta$. Obviously, if in $T$ the vertex $\alpha$ is reachable from the vertex $\gamma$, then $\beta$ is also reachable from it, and the path from $\gamma$ to $\alpha$ is the same in $ T_\alpha$, and $T_\beta$. Thus, $T_\alpha$ is a subtree of the tree $T_\beta$.
\end{proof}

\begin{assertion}
Let $T$ be a tree and 
${\cal B}\subset{\cal V}T$, $|{\cal B}|>1$. Then there is at least one vertex 
$\beta\in{\cal B}$ from which no vertex from ${\cal B}\setminus \{\beta\}$ is reachable. 
\end{assertion}

\begin{proof}
If the root of the tree $T$ (let it be a vertex $\delta$ for definiteness) is contained in ${\cal B}$, then the situation is trivial. From the root no other vertex is reachable. If $\delta\notin{\cal B}$, take an arbitrary vertex $\gamma\in{\cal B}$. Consider a path from $\gamma$ to $\delta$. It can leave the set ${\cal B}$ and enter it. Let $\beta$ be the last vertex of this path such that $\beta\in{\cal B}$, and  the arc originating from it not belong to ${\cal B}$. Then $\beta$ is the desired vertex, since it is impossible to get from it to ${\cal B}\setminus \{ \beta\}$. 
\end{proof}

\begin{assertion}
Let $T=T_\delta$ be a tree, ${\cal D}\subset{\cal V}T$, $\beta$ is a vertex from which the set ${\cal D}$ is unreachable. Then the path from $\beta$ to $\delta$ does not depend on arcs outgoing from the vertices of the set ${\cal D}$.
\end{assertion}
\begin{proof}
Obviously. We only note that even if  arcs outgoing from ${\cal D}$ are changed so that the resulting graph ceases to be a tree, this will not in any way affect the path from $\beta$ to $\delta$. 
\end{proof}
\begin{remark}
The formulation of the last statement assumes that the vertex $\beta$ itself is not contained in ${\cal D}$, since any vertex is reachable from itself. 
\end{remark}

\subsection{Tree partition of a set of tree (forest) vertices}

Let $T=T_\delta$ be a tree rooted at $\delta$, and let some special partition ${\mathfrak N}^T$ of its vertex set ${\cal V}T$ be given such that if $ {\cal X}\in {\mathfrak N}^T $, then 
$T|_{\cal X}$ is a tree. Let us call such a partition a tree partition of the set of tree vertices.

Any tree partition can be easily constructed by selecting a certain subset ${\cal A}$ of the set of vertices, including the root $\delta$. Let us remove from the tree $T$ arcs outgoing from all vertices of this set except vertex $\delta$ (no arc outgoes from the root). In this section, we denote the resulting auxiliary forest as $Q$.  The forest $Q$ consists of $|{\cal A}|$ trees $T^Q_\alpha$ with roots $\alpha\in{\cal A}$ (${\cal A}={\cal K}_Q $). The sets of vertices ${\cal V}T^Q_\alpha$ constitute the required partition (see Fig. \ref{treelike}).  
The partition itself is generated by the set ${\cal A}$. Let us denote it as $\mathfrak{N}_{\cal A}^T=\{ {\cal V}T^Q_\alpha|\alpha\in{\cal A}\}$. 

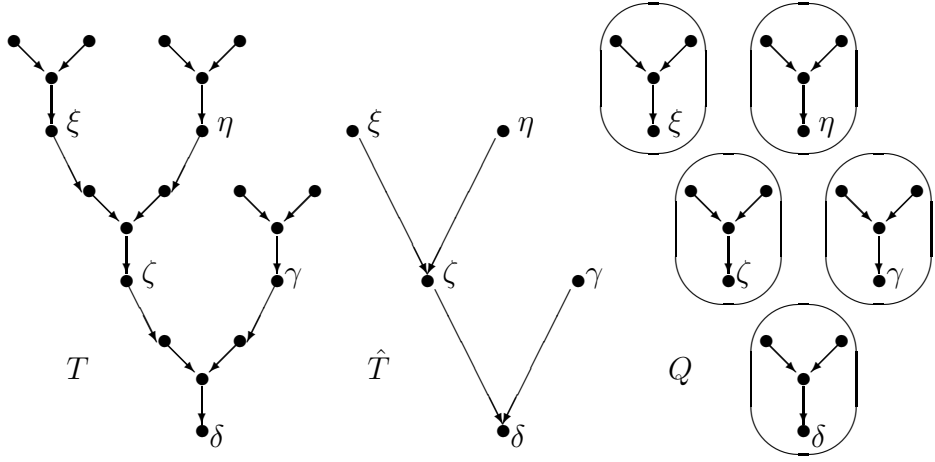
\begin{figure}[h]
\unitlength=1mm
\begin{center}
\begin{picture}(130,59)

\put(28,3){$\bullet$}
\put(28,10){$\bullet$}
\put(23,15){$\bullet$}
\put(33,15){$\bullet$}

\put(28,43){$\bullet$}
\put(28,50){$\bullet$}
\put(23,55){$\bullet$}
\put(33,55){$\bullet$}

\put(18,23){$\bullet$}
\put(18,30){$\bullet$}
\put(13,35){$\bullet$}
\put(23,35){$\bullet$}

\put(8,43){$\bullet$}
\put(8,50){$\bullet$}
\put(3,55){$\bullet$}
\put(13,55){$\bullet$}

\put(38,23){$\bullet$}
\put(38,30){$\bullet$}
\put(33,35){$\bullet$}
\put(43,35){$\bullet$}

\put(29,10){\vector(0,-1){5}}
\put(24,16){\vector(1,-1){4}}
\put(34,16){\vector(-1,-1){4}}

\put(29,50){\vector(0,-1){5}}
\put(24,56){\vector(1,-1){4}}
\put(34,56){\vector(-1,-1){4}}

\put(19,30){\vector(0,-1){5}}
\put(14,36){\vector(1,-1){4}}
\put(24,36){\vector(-1,-1){4}}

\put(39,30){\vector(0,-1){5}}
\put(34,36){\vector(1,-1){4}}
\put(44,36){\vector(-1,-1){4}}

\put(9,50){\vector(0,-1){5}}
\put(4,56){\vector(1,-1){4}}
\put(14,56){\vector(-1,-1){4}}

\put(29,44){\vector(-1,-2){4}}
\put(9,44){\vector(1,-2){4}}
\put(39,24){\vector(-1,-2){4}}
\put(19,24){\vector(1,-2){4}}

\put(108,3){$\bullet$}
\put(108,10){$\bullet$}
\put(103,15){$\bullet$}
\put(113,15){$\bullet$}

\put(108,43){$\bullet$}
\put(108,50){$\bullet$}
\put(103,55){$\bullet$}
\put(113,55){$\bullet$}

\put(98,23){$\bullet$}
\put(98,30){$\bullet$}
\put(93,35){$\bullet$}
\put(103,35){$\bullet$}

\put(88,43){$\bullet$}
\put(88,50){$\bullet$}
\put(83,55){$\bullet$}
\put(93,55){$\bullet$}

\put(118,23){$\bullet$}
\put(118,30){$\bullet$}
\put(113,35){$\bullet$}
\put(123,35){$\bullet$}

\put(109,10){\vector(0,-1){5}}
\put(104,16){\vector(1,-1){4}}
\put(114,16){\vector(-1,-1){4}}
\put(109,11){\oval(14,20)}

\put(109,50){\vector(0,-1){5}}
\put(104,56){\vector(1,-1){4}}
\put(114,56){\vector(-1,-1){4}}
\put(109,51){\oval(14,20)}

\put(99,30){\vector(0,-1){5}}
\put(94,36){\vector(1,-1){4}}
\put(104,36){\vector(-1,-1){4}}
\put(99,31){\oval(14,20)}

\put(119,30){\vector(0,-1){5}}
\put(114,36){\vector(1,-1){4}}
\put(124,36){\vector(-1,-1){4}}
\put(119,31){\oval(14,20)}

\put(89,50){\vector(0,-1){5}}
\put(84,56){\vector(1,-1){4}}
\put(94,56){\vector(-1,-1){4}}
\put(89,51){\oval(14,20)}

\put(68,3){$\bullet$}
\put(68,43){$\bullet$}
\put(48,43){$\bullet$}
\put(58,23){$\bullet$}
\put(78,23){$\bullet$}

\put(68,43){\vector(-1,-2){9}}
\put(50,43){\vector(1,-2){9}}
\put(78,23){\vector(-1,-2){9}}
\put(60,23){\vector(1,-2){9}}

\put(30,2){$\delta$}
\put(70,2){$\delta$}
\put(110,2){$\delta$} 
\put(21,24){$\zeta$}
\put(61,24){$\zeta$}
\put(100,24){$\zeta$}
\put(40,24){$\gamma$}
\put(80,24){$\gamma$}
\put(120,24){$\gamma$}
\put(11,44){$\xi$}
\put(51,44){$\xi$}
\put(91,44){$\xi$}
\put(31,44){$\eta$}
\put(71,44){$\eta$}
\put(111,44){$\eta$}

\put(11,11){$T$}
\put(51,11){$\hat T$}
\put(91,11){$Q$}

\end{picture} 
\caption{\small Constructing from a tree $T=T_\delta$ a tree $\hat T$ and a forest $Q$, the sets of vertices of which form a tree partition. It is determined by a given set of roots ${\cal A}={\cal K}_Q=\{ \delta,\gamma,\zeta,\eta,\xi\}$.}
\label{treelike}
\end{center}
\end{figure}

Along with $T^Q_\alpha$ trees, we will also consider $T_\beta$ trees ($T_\beta$ is the inclusion-maximal subtree of the tree $T$ rooted at the vertex $\beta$).

\begin{lemma}
Let $\mathfrak{N}_{\cal A}^T $ be a tree partition of the set of vertices of the tree $T$, and $\beta\in{\cal A}$.  Then for any ${\cal X}\in \mathfrak{N}_{\cal A}^T $ either ${\cal X}\cap{ \cal V}T_\beta=\emptyset$, or ${\cal X}\cap{\cal V}T_\beta={\cal X}$.
\end{lemma} 
\begin{proof}

The requirement $\beta\in{\cal A}$ is essential. Indeed, let $\beta\notin{\cal A}$. Since $Q$ is a forest obtained from $T$ by removing all arcs coming from the vertices of the set ${\cal A}$, then in this case it contains an arc coming from $\beta$. For definiteness, let this arc be $(\beta,\gamma)$. It is contained in some component of the graph $Q$ and, therefore, the set of vertices ${\cal X}$ of this component contains the vertex $\beta$ (as well as the vertex $\gamma$).
But the arc $(\beta,\gamma)$ does not belong to the tree $T_{\beta}$, which means that the vertex $\gamma$ does not belong to it, but the vertex $\beta$ does. Thus, ${\cal X}\in \mathfrak{N}_{\cal A}^T $, while $ \beta\in {\cal X}\cap{\cal V}T_\beta\neq \emptyset $, however $\gamma \notin {\cal X}\cap{\cal V}T_\beta \subsetneq {\cal X}\ni \gamma$. 

If $\beta\in{\cal A}$, then no arc in the forest $Q$ comes from it and the required result is guaranteed. Let's consider this situation in more detail.
   
For an arbitrary ${\cal X}\in{\mathfrak{N}_{\cal A}^T}$, $Q|_{\cal X}=T|_{\cal X}$ and ${\cal X}={\cal V}T^Q_\alpha$ for some $\alpha\in{\cal K}_Q$.
If  
 $\alpha\in {\cal V}T_\beta$, then by Claims 1 and 2 $T_\alpha$ is a subtree of the tree $T_\beta$. Then $T^Q_\alpha$ is a subtree of the tree $T_\beta$. That is, ${\cal X}\cap{\cal V}T_\beta={\cal X}$.
Now, let $\alpha\notin {\cal V}T_\beta$. According to Claim 2, the tree $T_\alpha$ can, of course, contain the tree $T_\beta$ as a subtree (if $\beta\in{\cal V}T_\alpha$), but this does not change anything for the tree $T ^Q_\alpha$.  The trees $T^Q_\alpha$ and $T^Q_\beta$ do not intersect, since they are different components of the forest $Q$. This means $\beta\notin {\cal V}T^Q_\alpha$. Then, by Claim 1, ${\cal V}T^Q_\alpha\cap{\cal V}T_\beta=\emptyset$.
\end{proof}

We will further index the elements of the partition by the roots of the components of the auxiliary forest $Q$: $\mathfrak{N}_{\cal A}^T=\{  {\cal X}_\alpha|\alpha\in{\cal A}\}$, ${\cal X}_\alpha={\cal V}T^Q_\alpha$. 
 
Let us associate the tree $T$ with the graph $\hat T$ according to the rule. Its vertex set is the vertices $\alpha\in{\cal A}$ (${\cal V}\hat T={\cal A} $). A pair $(\zeta,\eta)$ is an arc of a graph $\hat T$ if the graph $T$ contains an arc with an outcome in 
${\cal X}_\zeta$ and entry in ${\cal X}_\eta$ (see Fig. \ref{treelike}). The arc corresponding to it was removed when constructing the partition $\mathfrak{N}_{\cal A}^T$. In the original tree $T$ it came from the vertex $\zeta$ and entered the set ${\cal X}_\eta$. Obviously, the graph $\hat T$ is a tree with the vertex set ${\cal A}$. The set ${\cal A}$ itself determines the tree $\hat T$ in a unique way from the original tree $T$. 
 
\begin{lemma} 
Reachability (unreachability) in tree $T$ of a set ${\cal X}_\alpha$ from ${\cal X}_\beta$ is equivalent to reachability (unreachability) of $\alpha$ from $\beta$ in tree $\hat T$.
\end{lemma} 

\begin{proof}
Obviously from the construction of the tree $\hat T$.  
It contains an arc $(\zeta,\eta)$ if and only if the forest $T$ has an arc with an exit in ${\cal X}_\zeta$ (namely outcome is $\zeta$) and an entry in ${\cal X}_\eta$.
\end{proof}

\begin{lemma}
Let $\mathfrak{N}_{\cal A}^T $ be a tree partition of the set of vertices of the tree $T$, $\mathfrak{M}\subset \mathfrak{N}_{\cal A}^T $, $|\mathfrak{M}|>1 $. Then there is at least one ${\cal X}\in\mathfrak{M}$ from which none of the other elements of $\mathfrak{M}$ in the tree $T$ is    reachable.  
\end{lemma} 

\begin{proof}
The family $\mathfrak{M}$ can be represented as $\mathfrak{M}=\{ {\cal X}_\alpha|\alpha\in{\cal B}\}$, where ${\cal B} \subset{\cal A}$.  
The set ${\cal A}$ is the set of vertices of the tree $\hat{T}$. By Claim 3, there is a vertex $\beta\in{\cal B}$ such that from it in the tree $\hat{T}$ not a single remaining vertex from ${\cal B}$ is reachable. By Lemma 6, for the tree $T$ itself, this means that from ${\cal X}_\beta$ none of the sets from $\mathfrak{M}\setminus\{{\cal X}_\beta\}$ is reachable.   

\end{proof}
 
\begin{remark} All statements of this paragraph remain valid if in the formulations the original tree $T$ is replaced by forest $F$, with the proviso that when constructing a tree partition $\mathfrak{N}_{\cal A}^F$ 
(forest $F$ is matched by the same rule to forest $\hat{F}$)  
the set ${\cal A}$ must include all the roots of the forest $F$.  
\end{remark}

\subsection{Forest required}

Let (\ref{nonequal}) hold, $F\in\tilde{\cal F}^k$, $\delta\in{\cal K}_F$, and $T^F_\delta$ is  the  tree  of $F$ rooted at vertex $\delta$.
For each atom ${\cal W}$ from the partition $\aleph_k$
the graph $F|_{\cal W}$ is obviously a forest (that it is a tree remains to be seen).The atom ${\cal W}$ itself is in turn divided into sets of vertices ${\cal X}_\alpha$ of connected components of the graph $F|_{\cal W}$ with roots at the vertices $\alpha$. The set ${\cal A}$ of all such roots $\alpha$ over all ${\cal W}\in\aleph_k$ obviously includes all the roots of the forest $F$ itself. Thus ${\cal A}$ generates a tree partition $\mathfrak{N}_{\cal A}^F$ of the forest $F$: $\mathfrak{N}_{\cal A}^F=\{ { \cal X}_\alpha|\alpha\in{\cal A}\}$.  If we need to indicate which atom ${\cal W}$ a specific set ${\cal X}_\alpha$ belongs to, we will mark it with a superscript: ${\cal X}_\alpha^{\cal W}$.

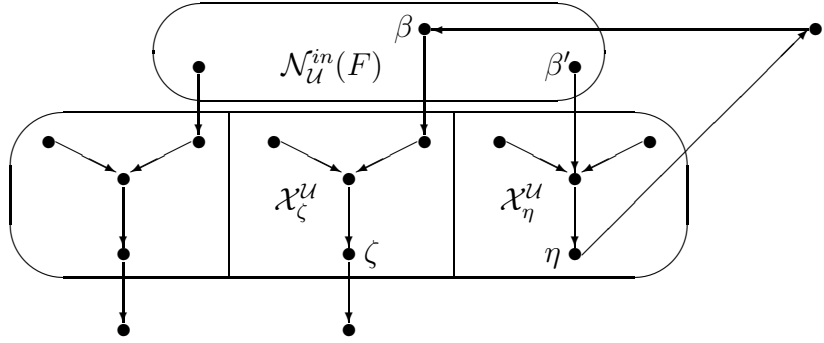
\begin{figure}[h]
\unitlength=1mm
\begin{center}
\begin{picture}(110,44)

\put(26,35){\vector(0,-1){8}}
\put(16,20){\vector(0,-1){8}}
\put(16,10){\vector(0,-1){8}}
\put(7,26){\vector(2,-1){8}}
\put(25,26){\vector(-2,-1){8}}
\put(15,20){$\bullet$}
\put(15,10){$\bullet$}
\put(15,0){$\bullet$}
\put(5,25){$\bullet$}
\put(25,25){$\bullet$}
\put(25,35){$\bullet$}

\put(56,40){\vector(0,-1){13}}
\put(46,20){\vector(0,-1){8}}
\put(46,10){\vector(0,-1){8}}
\put(37,26){\vector(2,-1){8}}
\put(55,26){\vector(-2,-1){8}}
\put(45,20){$\bullet$}
\put(45,10){$\bullet$}
\put(45,0){$\bullet$}
\put(35,25){$\bullet$}
\put(55,25){$\bullet$}
\put(55,40){$\bullet$}

\put(76,35){\vector(0,-1){13}}
\put(76,20){\vector(0,-1){8}}
\put(77,11){\vector(1,1){30}}
\put(67,26){\vector(2,-1){8}}
\put(85,26){\vector(-2,-1){8}}
\put(107,41){\vector(-1,0){50}}

\put(75,20){$\bullet$}
\put(75,10){$\bullet$}
\put(107,40){$\bullet$}
\put(65,25){$\bullet$}
\put(85,25){$\bullet$}
\put(75,35){$\bullet$}

\put(48,10){$\zeta$}
\put(72,10){$\eta$}
\put(52,40){$\beta$}
\put(72,35){$\beta'$}
\put(36,17){${\cal X}_\zeta^{\cal U}$}
\put(66,17){${\cal X}_\eta^{\cal U}$}

\put(50,38){\oval(60,13)}
\put(46,19){\oval(90,22)}
\put(30,30){\line(0,-1){22}}
\put(60,30){\line(0,-1){22}}
\put(37,35){${\cal N}^{in}_{\cal U}(F)$}

\end{picture} 
\caption{\small Shown  arcs coming from the vertices of  atom ${\cal U}$ and entering it in forest $F$ and one of  arcs entering ${\cal N}^{in}_{\cal U}(F)$. A situation is shown (not realized in reality) when the set of vertices of one connected component of the graph $F|_{\cal U}$ is reachable from another (${\cal X}^{\cal U}_\zeta$ from ${ \cal X}^{\cal U}_\eta$) through vertices not belonging to ${\cal U}$. All vertices of the set ${\cal N}^{in}_{\cal U}(F)$ belong to different atoms.}
\label{nin}
\end{center}
\end{figure}

\begin{theorem}
Let (\ref{nonequal}) hold, and $F\in\tilde{\cal F}^k$,  ${\cal U}\in\aleph_k$, $\delta\in{\cal K}_F$, ${\cal U}\subset {\cal V}T^F_\delta$. Then there exists a set ${\cal D}\subset {\cal V}T^F_\delta$ and a forest $G\in\tilde{\cal F}^k$ such that  
 arcs coming from vertices of  $\overline{\cal D}$ in  $G$ and $F$ coincide,
${\cal U}\subset{\cal V}T^G_\delta$, ${\cal N}^{in}_{\cal U}(G)=\emptyset$.  

\end{theorem}

\begin{proof}
Let us present the required set:
\begin{equation}
{\cal D}=\underset{\beta\in{\cal N}^{in}_{\cal U}(F)}{\cup}{\cal V}T^F_\beta \ . 
\label{diz}
\end{equation} 
Here $T^F_\beta$ is the inclusion-maximal subtree of  $F$ (or, what is the same, of the tree $T^F_\delta$) rooted at $\beta$.

First of all, note that according to Corollary 1 from Lemma 3, from any atom $W$ at most one arc $(\beta,\gamma)$ can come and end in ${\cal U}$. That is, all vertices from ${\cal N}^{in}_{\cal U}(F)$ belong to different atoms. We assume that an unlabeled atom ${\cal U}$ (like any other unlabeled atom ${\cal W}$) can consist of several sets ${\cal X}^{\cal U}_\beta$ of family $\mathfrak{N}_{\cal A}^F$. Therefore, we do not exclude the possibility that some set ${\cal X}^{\cal U}_\zeta$ can be reachable in $F$ from another set ${\cal X}^{\cal U }_{\eta}$ through a sequence of arcs, including the vertices of other atoms (see Fig. \ref{nin}).This means that formally (\ref{diz}) is not a disjunctive union.\footnote{By disjoint union we mean the union of a family of pairwise disjoint sets.}

For any vertices $\beta$ and $\beta'$, by Claim 2, either ${\cal V}T^F_\beta\cap{\cal V}T^F_{\beta'}=\emptyset$, or one of these sets is a subset of another. Therefore (\ref{diz}) can be represented as disjunctive. To do this, it is necessary to leave in the union only trees with roots $\beta$ from which the set ${\cal N}^{in}_{\cal U}(F)\setminus \{\beta\}$ is unreachable. According to Claim 3, the set of such vertices is not empty if $|{\cal N}^{in}_{\cal U}(F)|>1$. In the example in Fig. \ref{nin} the vertex $\beta'$ (and the set ${\cal V}T^F_{\beta'}$) does not participate in the disjoint union representation (\ref{diz}), since from $\beta' $ vertex $\beta$ is reachable.

Further, any of  sets ${\cal V}T^F_\beta$ for $\beta\in{\cal N}^{in}_{\cal U}(F)$ (and the last circumstance is important), 
is a disjoint union of some sets from the tree partition $\mathfrak{N}_{\cal A}^F$. Indeed, if $\beta\in{\cal N}^{in}_{\cal U}(F)$, then $\beta$ belongs to some unlabeled atom ${\cal W}$ and is the root of a connected component of the graph $F|_{\cal W}$. That is 
$\beta\in{\cal A}$. Then, by Lemma 5, for any ${\cal X}\in \mathfrak{N}_{\cal A}^F $ or ${\cal X}\cap{\cal V}T_\beta=\emptyset$, or ${\cal X}\cap{\cal V}T_\beta={\cal X}$. But then the set ${\cal D}$ is also a disjoint union of some set of sets from $\mathfrak{N}_{\cal A}^F$. We do not assume that if ${\cal W}\cap{\cal D}\neq \emptyset$, then the atom ${\cal W}$ falls into ${\cal D}$ completely. It is important that some of its subsets contained in $\mathfrak{N}_{\cal A}^F$ completely fall there. Now, since for any set from $\mathfrak{N}_{\cal A}^F$ by Theorem 3 (\ref{gxfx}) is satisfied, then   

\begin{equation}
\Upsilon^H_{\cal D}=\Upsilon^F_{\cal D}
\label{FGZ}
\end{equation}
for any $H\in\tilde{\cal F}^k$.

According to Property 9, there is a forest $H\in\tilde{\cal F}^k$ such that  ${\cal N}^{in}_{\cal U}(H)=\emptyset$ and it contains arcs, coming in  $H$ from the vertices of the atom ${\cal U}$, the same as in $F$ (at this point in the proof we do not even care which arcs in $H$ come from the vertices of the set ${\cal U }$). Consider the graph $G=F^H_{\uparrow{\cal D}}$.  Since ${\cal N}^{in}_{\cal D}(F)=\emptyset$, then by Lemma 1 $G$ is a forest. And since both $F$ and $H$ have arcs coming from all vertices of the set ${\cal D}$, then $G\in{\cal F}^k$. In this case, (\ref{FGZ}) and $\Upsilon^H_{\cal D}=\Upsilon^G_{\cal D}$ are satisfied. Thus

\begin{equation*}
\Upsilon^G=\Upsilon^G_{\cal D}+\Upsilon^G_{\overline{\cal D}}= \Upsilon^H_{\cal D}+\Upsilon^F_{\overline{\cal D}}=\Upsilon^F_{\cal D}+\Upsilon^F_{\overline{\cal D}}=\Upsilon^F=\varphi^k \ , 
\end{equation*}
whence we conclude that $G\in\tilde{\cal F}^k$. By construction, ${\cal N}^{in}_{\cal U}(G)=\emptyset$. Arcs coming from vertices of the set $\overline{\cal D}$ remained the same in the forest $G$, that is, the same as in $F$.
The vertex $\delta$ remains a root in the forest $G$ and $T^G_\delta$ is a connected component of this forest.

It is important to note here that the replacement of arcs concerns only a certain number of vertices of a single forest tree $F$ --- tree $T^F_\delta$. This means that the sets of vertices of other trees in the forest $F$ can only increase with such an operation. That is, for any atom ${\cal E}$, if ${\cal E}\cap{\cal V}T^F_\delta=\emptyset$, then ${\cal E}\cap{\cal V}T^G_\delta=\emptyset$.

Let us now find out where in the forest $G$ the atom ${\cal U}$ itself is located.  The set ${\cal U}$ consists of a certain number of sets ${\cal X}^{\cal U}_\alpha$ of the tree partition $\mathfrak{N}_{\cal A}^F$. By Lemma 7, among them there is at least one set (let, for definiteness, ${\cal X}^{\cal U}_\zeta$), from which it is unreachable in the forest $F$ (in the tree $T^F_\delta$ ) none of the remaining sets ${\cal X}^{\cal U}_\alpha$. It is also impossible to get into the set ${\cal X}^{\cal U}_\zeta$ by moving along arcs starting at $\zeta$ (otherwise there would be a contour including $\zeta$).  
Then the set ${\cal D}$ is also unreachable from the set ${\cal X}^{\cal U}_\zeta$ (or, what is the same, from the vertex $\zeta$) and $\zeta\notin  {\cal D}$. Then, by Claim 4, the path from $\zeta$ to the root of the tree $\delta$ does not depend on the arcs coming from the vertices of the set ${\cal D}$. And from this in turn it follows that $\zeta\in{\cal V}T^G_\delta$. Since the atom is indivisible, the entire set ${\cal U}$ is contained in the tree $T^G_\delta$. 
\end{proof}

\begin{remark}
 As a result, this theorem still does not exclude that $F|_{\cal U}$ can consist of several components. And it also allows for a situation where some set ${\cal X}^{\cal U}_\zeta$ can be reachable in $F$ from another set ${\cal X}^{\cal U}_ {\eta}$ through a sequence of arcs including vertices of atoms other than ${\cal U}$. This means, in particular, that the set ${\cal X}^{\cal U}_{\eta}$ can be contained in the set ${\cal D}$, that is, that ${\cal U}\cap{\cal D}\neq \emptyset$. Thus, formally, some of the arcs coming from the vertices of the set ${\cal U}$ may turn out to be different in  $G$ than in  $F$. But for our purposes now something else is important. Namely: if some atom ${\cal E}$ was in the forest $F$ not in the same tree with ${\cal U}$, then in the constructed forest $G$, in which the arcs in ${\cal U }$ do not enter, the atom ${\cal E}$ is still not in the same tree with ${\cal U}$.  
\end{remark}

\section{Full justification of the hypothesis} 
\subsection{Main theorem}
Now the hypothesis put forward can be formulated as a theorem. 

\begin{theorem}
Any forest $F\in\tilde{\cal F}^k$ restricted to an atom of algebra $\mathfrak{A}_k$ is a tree.
\end{theorem}

\begin{proof}
According to Theorem 1, it is sufficient to consider the situation when (\ref{nonequal}) is satisfied. According to Property 15, only unlabeled atoms are subject to verification.

We will carry out the proof by contradiction. Suppose there is an unlabeled atom ${\cal U}$ such that the graph $F|_{\cal U}$ is disconnected. It is enough to assume that it consists of two components. Then arcs coming from them in $F$ enter some atoms ${\cal E}$ and ${\cal S}$. By Corollary 1 from Lemma 3, atoms ${\cal E}$ and ${\cal S}$ cannot coincide. And by Lemma 3 itself, in any forest $G\in\tilde{\cal F}^k$ only one of these atoms can be in not the same tree with ${\cal U}$. Then, by Lemma 2, there must exist both a forest in which ${\cal E}$ is not in the same tree with ${\cal U}$, and a forest in which ${\cal S}$ is not in the same tree with ${ \cal U}$, as shown in Fig. \ref{hip} (by the way, according to Theorem 2, none of these atoms can be labeled). For definiteness, consider the forest $G$ (Fig. \ref{hip} on the left), in which ${\cal E}$ is not in the same tree as ${\cal U}$.  According to Theorem 4, without loss of generality we can assume that in  $G$ no arcs enter ${\cal U}$, and the atom ${\cal E}$ is still not in the same tree with ${\cal U }$. As for the atom ${\cal S}$, then by Lemma 3 there is nowhere for it to be except in the same tree as ${\cal U}$.  
Thus, we are in the conditions of Lemma 4, according to which only one arc can originate in  $F$ from ${\cal U}$ to the set ${\cal E}\cup{\cal S}$. Contradiction.
\end{proof} 

Theorem 5 allows us to strengthen Property 9, which is constantly used in proofs when constructing forests, by combining it with Theorem 4.

\begin{proposition}
Let (\ref{nonequal}) hold,  $F\in\tilde{\cal F}^k$,  ${\cal U}\in\aleph_k$ and    
\begin{equation}
{\cal D}=\underset{\beta\in{\cal N}^{in}_{\cal U}(F)}{\cup}{\cal V}T^F_\beta \ . 
\label{cup}
\end{equation}
Then  

i) ${\cal D}\cap {\cal U}=\emptyset$; 

ii) union (\ref{cup}) is disjunctive;

iii) there is a forest $G\in\tilde{\cal F}^k$ such that  
 arcs coming from  vertices of the set $\overline{\cal D}$ in  $G$ and $F$ coincide,  
and ${\cal N}^{in}_{\cal U}(G)=\emptyset$.  
\end{proposition}

\begin{proof}
$i)$ If ${\cal U}\in\aleph_k^\bullet$, then no arcs come from it and automatically ${\cal D}\cap {\cal U}=\emptyset$. If ${\cal U}\in\aleph_k^\circ$, then by Theorem 5 in $F$ there is only one arc coming from ${\cal U}$, and since $F|_{\cal U}$ --- tree, then the set ${\cal N}_{\cal U}^{in}(F)$ is unreachable from ${\cal U}$ (otherwise there would be a contour in $F$). But then ${\cal D}\cap {\cal U}=\emptyset$. 

$ii)$ By Theorem 5, $F|_{\cal U}$ is a tree rooted at some vertex $\zeta$. This means that in  $F$ the vertex $\zeta$ is reachable from all vertices of the set ${\cal N}_{\cal U}^{in}(F)$ and none of the vertices ${\cal N}_{\cal U}^{in}(F)$, is reachable from $\zeta$ (otherwise there would be a contour including $\zeta$). But then any vertex $\beta\in{\cal N}_{\cal U}^{in}(F)$ is unreachable from any vertex of the set ${\cal N}_{\cal U}^{in }(F)\setminus \{\beta \}$. Then, by Claim 2, if  $\{\beta,\beta'\} \subset{\cal N}_{\cal U}^{in}(F)$, then  ${\cal V}T_\beta^F\cap {\cal V}T^F_{\beta'}=\emptyset$.

$iii)$ By Property 9, there is a forest $H\in\tilde{\cal F}^k$ such that ${\cal N}^{in}_{\cal U}(H)=\emptyset$. By Theorem 4, $G=F^H_{\uparrow{\cal D}}$ is the required forest. 
\end{proof}

The essence of Proposition 2 is as follows. Let ${\cal D}$ be the entire set of vertices from which some atom ${\cal U}$ in $F\in\tilde{\cal F}^k$ is reachable. Without violating the minimality property, this set can be “detached” \ from ${\cal U}$ and redistributed among the trees of the forest. Arcs coming from vertices of the set $\overline{\cal D}$ remain the same (while ${\cal U}\subset\overline{\cal D}$). For the set ${\cal D}$ itself, the representation (\ref{diz}) is a disjoint union. 

Note that Theorem 5 can be reformulated as follows.  

{\bf Theorem 5'.} {\it The family of atoms $\aleph_k$ is a tree partition for any $F\in\tilde{\cal F}^k$.} 

Let us clarify this formulation.  Let us define the set ${\cal A}$ for the forest $F\in\tilde{\cal F}^k$ according to the rule: $\alpha\in{\cal A}$ if there is ${\cal X}\in \aleph_k$ such that $\alpha$ is the root of the tree $F|_{\cal X}$. It is obvious that $|{\cal A}|=|\aleph_k|$.   Atoms themselves that make up the family $\aleph_k$ can be indexed by vertices from ${\cal A}$: $\aleph_k= \mathfrak{N}^F_{\cal A}=\{{\cal X}_\alpha| \alpha\in{\cal A} \} $.  Similarly, for any other forest $F'\in\tilde{\cal F}^k$ we can define the set ${\cal A}'$ and index the same atoms in a different way and write the tree partition for $F'$ through them.  At the same time, the sets     
${\cal A}$ and ${\cal A}'$ may not only not coincide, but may not even intersect. 

\begin{figure}[h]
\unitlength=1mm
\begin{center}
\begin{picture}(105,43)

\put(0,28){$V$}
\put(18,28){$3$}
\put(6,12){$3$}
\put(12,0){$2$}
\put(12,7){$2$}
\put(12,16){$1$}
\put(12,23){$1$}
\put(12,32){$2$}
\put(12,39){$2$}

\put(6,4){\vector(1,0){14}}
\put(20,6){\vector(-1,0){14}}
\put(6,20){\vector(1,0){14}}
\put(20,22){\vector(-1,0){14}}
\put(6,36){\vector(1,0){14}}
\put(20,38){\vector(-1,0){14}}
\put(5,20){\vector(0,-1){14}}
\put(21,22){\vector(0,1){14}}
\put(4,4){$\bullet$}
\put(4,20){$\bullet$}
\put(4,36){$\bullet$}
\put(20,4){$\bullet$}
\put(20,20){$\bullet$}
\put(20,36){$\bullet$}
\put(1,4){$\gamma$}
\put(23,4){$\xi$}
\put(1,20){$\beta$}
\put(23,20){$\eta$}
\put(1,35){$\alpha$}
\put(23,35){$\zeta$}

\put(38,28){$F$}
\put(51,33){${\cal L}$}
\put(51,17){${\cal U}$}
\put(51,1){${\cal M}$}

\put(45,20){\vector(0,-1){14}}
\put(60,5){\vector(-1,0){14}}
\put(60,21){\vector(-1,0){14}}
\put(60,37){\vector(-1,0){14}}

\put(44,4){$\bullet$}
\put(44,20){$\bullet$}
\put(44,36){$\bullet$}
\put(60,4){$\bullet$}
\put(60,20){$\bullet$}
\put(60,36){$\bullet$}
\put(53,5){\oval(26,10)}
\put(53,21){\oval(26,10)}
\put(53,37){\oval(26,10)}
\put(41,4){$\gamma$}
\put(63,4){$\xi$}
\put(41,20){$\beta$}
\put(63,20){$\eta$}
\put(41,35){$\alpha$}
\put(63,35){$\zeta$}

\put(77,28){$F'$}
\put(91,33){${\cal L}$}
\put(91,17){${\cal U}$}
\put(91,1){${\cal M}$}

\put(84,4){$\bullet$}
\put(84,20){$\bullet$}
\put(84,36){$\bullet$}
\put(100,4){$\bullet$}
\put(100,20){$\bullet$}
\put(100,36){$\bullet$}
\put(93,5){\oval(26,10)}
\put(93,21){\oval(26,10)}
\put(93,37){\oval(26,10)}
\put(81,4){$\gamma$}
\put(103,4){$\xi$}
\put(81,20){$\beta$}
\put(103,20){$\eta$}
\put(81,35){$\alpha$}
\put(103,35){$\zeta$}
\put(101,22){\vector(0,1){14}}
\put(86,5){\vector(1,0){14}}
\put(86,21){\vector(1,0){14}}
\put(86,37){\vector(1,0){14}}
\end{picture} 
\caption{\small For the graph $V$ (on the left) the forests $F,F'\in\tilde{\cal F}^2$. The sets ${\cal L},{\cal U},{\cal M}$ constitute both the family $\aleph_2$ and the family $\aleph_3$, with the only difference that ${\cal U}\in \aleph_2^\circ$, but ${\cal U}\in\aleph_3^\bullet$. The forest $F$ restricted to any of the atoms does not coincide with the corresponding restriction of the forest $F'$.}
\label{woody}
\end{center}
\end{figure}
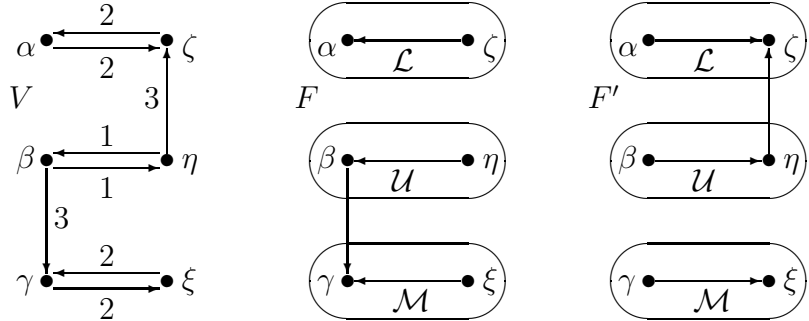

\begin{example}

For the original graph $V$ (Fig.\ref{woody}), the family $\aleph_2$ consists of three atoms: two labeled ${\cal L}$ and ${\cal M}$, and one unlabeled ${\cal U}$.  For the forest $F\in\tilde{\cal F}^2$ the set ${\cal A}$ is $\{\alpha,\beta,\gamma\}$. For a forest $F'$ the corresponding set is ${\cal A}'=\{\zeta,\eta,\xi\}$. The tree partitions themselves for $F$ and $F'$ can be written as:
$\mathfrak{N}^F_{\cal A}=\{{\cal X}_\alpha, {\cal X}_\beta, {\cal X}_\gamma\}$ and $\mathfrak{N}^{F'}_{{\cal A}'}=\{{\cal X}_\zeta, {\cal X}_\eta, {\cal X}_\xi\}$, where  $ {\cal X}_\alpha={\cal X}_\zeta={\cal L}$, $ {\cal X}_\beta={\cal X}_\eta={\cal U}$ and $ {\cal X}_\gamma={\cal X}_\xi={\cal M}$. At the same time, of course,  
$\aleph_2= \mathfrak{N}^F_{\cal A}= \mathfrak{N}^{F'}_{{\cal A}'}$. Note that in this example there is no spanning tree. Therefore $\varphi^1=\infty$ and
\begin{equation*}
\varphi^1-\varphi^2 >\varphi^2-\varphi^3 \ . 
\end{equation*}
Nevertheless, the algebras $\mathfrak{A}_2$ and $\mathfrak{A}_3$ turned out to be the same, as well as the families of atoms. But at the same time, the families of labeled and unlabeled atoms differ: $\aleph_2=$ $\aleph_3$ $=\aleph_3^\bullet$ $=\{{\cal L},{\cal U},{\cal M}\}$, but  $\aleph_2^\bullet=$ $\{{\cal L},{\cal M}\}$ and $\aleph_2^\circ=$ $\{{\cal U}\}$. 
\end{example}

\subsection{Narrowing the forest to an atom of someone else's algebra}

After proof of Theorem 5, the conditional conclusion formulated in \cite{V7} turns into a theorem. 

\begin{theorem}
The forest $F\in\tilde{\cal F}^{k-1}$ restricted to an atom of algebra $\mathfrak{A}_k$ is a tree.
\end{theorem}

\begin{proof}
Indeed, if equality (\ref{equal}) is satisfied, then from Property 5 (\ref{al})

\begin{equation*}
\tilde{\cal F}^{k-1}|_{\cal E}\subseteq \tilde{\cal F}^k|_{\cal E}
,  \ \  {\cal E}\in\aleph_k. 
\end{equation*}
If the strict inequality (\ref{nonequal}) holds, then by Property 16, if ${\cal M}$ is a labeled atom, then $F|_{\cal M}$ is a tree. For unlabeled atoms, according to Property 17 (\ref{uat})

\begin{equation*}
\tilde{\cal F}^{k-1}|_{\cal U}\subseteq \tilde{\cal F}^k|_{\cal U},  \ \  {\cal U}\in\aleph_k^\circ. 
\end{equation*}
\end{proof}

Thus, the family of atoms $\aleph_k$ turns out to be a tree partition for any forest $F\in\tilde{\cal F}^k\cup\tilde{\cal F}^{k-1}$ in the sense that $ F|_{\cal U}$ --- tree for ${\cal U}\in\aleph_k$. This greatly facilitates the construction of forests from $\tilde{\cal F}^{k-1}$ and a family of atoms $\aleph_{k-1}$ with known $\tilde{\cal F}^{k}$ and $ \aleph_{k}$. 

Theorem 6 provokes the question: maybe the restriction of  $F\in\tilde{\cal F}^{k-1}$ to an atom of algebra $\mathfrak{A}_{k+1}$ is a tree? The answer is no. The reason is as follows. When (\ref{nonequal}), if ${\cal M}\in\aleph_k^\bullet$,  although the set $\tilde{\cal F}^{k-1}|_{\cal M}$ consists of trees, but these trees may not belong to the set $\tilde{\cal F}^k|_{\cal M}$. Therefore, despite the fact that any tree from $\tilde{\cal F}^k|_{\cal M}$ restricted to an atom ${\cal L}$ of the algebra $\mathfrak{A}_{k+1 }$ is still a tree (if the atom ${\cal L}$ is a subset of the atom ${\cal M}$), then for trees of the set $\tilde{\cal F}^{k-1}| _{\cal M}$, this is, generally speaking, no longer the case. To demonstrate this effect, it is convenient to use the example from \cite{V7}.

\begin{example}
Weights of arcs of the graph $V$: $V_{ba}=1$, $V_{ac}=V_{bd}=2$, $V_{cb}=3$, (see figure\ref{ato}) . Here, each of the sets ${\cal F}^k$, $k\in\{1,2,3,4\}$ consists of a single forest: $\{F_k\}={\cal F}^{ k}=\tilde{\cal F}^{k}$. Forest $F_4$ --- an empty forest consisting of 4 vertices, not shown in the figure. We have: $\phi^4=0$, $\phi^3=1$, $\phi^2=3$, $\phi^1=7$ and by definition $\phi^0=\infty$. The convexity inequalities have the form 

\begin{equation*}
\phi^0-\phi^1>\phi^1-\phi^2>\phi^2-\phi^3>\phi^3-\phi^4 \ ,
\end{equation*}
in which the sign of strict inequality is applied everywhere.  
\begin{figure}[h]
\unitlength=1mm
\begin{center}
\begin{picture}(115,30)

\put (8,25){$V$}
\put(0,5){$\centerdot$}
\put(13,5){$\centerdot$}
\put(0,18){$\centerdot$}
\put(13,18){$\centerdot$}
\put(0,20){$a$}
\put(13,20){$b$}
\put(0,1){$c$}
\put(13,1){$d$}
\put(1,18){\vector(0,-1){12}}
\put(13,18){\vector(-1,0){12}}
\put(14,18){\vector(0,-1){12}}
\put(1,5){\vector(1,1){13}}
\put(7,19){1}
\put(2,12){2}
\put(8,9){3}
\put(15,11){2}

\put(25,25){$F_3, \aleph_3$}
\put(24,5){$\centerdot$}
\put(37,5){$\centerdot$}
\put(24,18){$\centerdot$}
\put(37,18){$\centerdot$}
\put(37,18){\vector(-1,0){12}}
\put(31,18){\oval(17,6)}
\put(25,5){\circle{5}}
\put(37,5){\circle{5}}

\put (49,25){$F_2,\aleph_2$}
\put(48,5){$\centerdot$}
\put(61,5){$\centerdot$}
\put(48,18){$\centerdot$}
\put(61,18){$\centerdot$}
\put(61,18){\vector(-1,0){12}}
\put(49,18){\vector(0,-1){12}}
\put(49,18){\oval(6,6)[tl]}
\put(46,18){\line(0,-1){12}}
\put(49,6){\oval(6,6)[bl]}
\put(49,21){\line(1,0){12}}
\put(61,18){\oval(6,6)[tr]}
\put(49,3){\line(1,1){15}}
\put(61,5){\circle{5}}

\put(73,25){$F_1,\aleph_1$}
\put(72,5){$\centerdot$}
\put(85,5){$\centerdot$}
\put(72,18){$\centerdot$}
\put(85,18){$\centerdot$}
\put(73,18){\vector(0,-1){12}}
\put(86,18){\vector(0,-1){12}}
\put(73,5){\vector(1,1){13}}
\put(79,12){\oval(20,20)}

\put (96,25){$F_1,\{a,b\}$}
\put(97,5){$\centerdot$}
\put(110,5){$\centerdot$}
\put(97,18){$\centerdot$}
\put(110,18){$\centerdot$}
\put(98,18){\vector(0,-1){12}}
\put(111,18){\vector(0,-1){12}}
\put(98,5){\vector(1,1){13}}
\put(104,18){\oval(17,6)}

\end{picture} 
\caption{\small The graph $V$ and  weights of its arcs (on the left), minimal forests $F_{3,2,1}$ and atoms of  algebras $\mathfrak{A}_{3,2,1}$. On the right, the forest $F_1$, restricted to the labeled atom $\{a,b\}$ of the subset algebra $\mathfrak{A}_3$, is not a tree.}
\label{ato}
\end{center}
\end{figure}
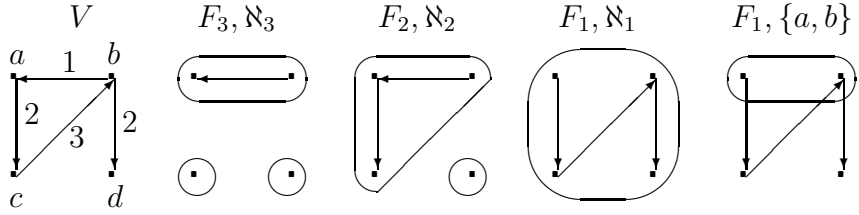

Since each of the sets $\tilde{\cal F}^k$, $k\in\{1,2,3,4\}$ consists of a single forest $F_k$, then the restriction of $F_k$ to any atom from $ \aleph_k$ is automatically a tree, since the connected components of the forest are trees. And the sets of tree vertices of $F_k$ constitute the family $\aleph_k$ in this situation.

In this example, the set $\{ a,b,c\}$ is a labeled atom of the algebra $\mathfrak{A}_2$. Consider the generated subgraph $F_1|_{\{ a,b,c\}}$. It is a tree (in full accordance with Theorem 6), which is conveniently described by two arcs: $(a,c)$ and $(c,b)$. But this tree is not contained in the set of trees ${\cal F}^2|_{\{ a,b,c\}}$, which consists of a single tree $F_2|_{\{ a,b,c\} }$ (its arcs: $(b,a)$ and $(a,c)$). This discrepancy is immediately reflected when the forest $F_1$ is narrowed to the atoms of the algebra $\mathfrak{A}_3$. The set $\{a,b\}$ is a labeled atom of this algebra and $\{a,b\} \subset \{ a,b,c\}$. Again, in full accordance with Theorem 6, $F_2|_{\{a,b\}}$ is a tree. It consists of a single arc $(b,a)$. But the induced subgraph $F_1|_{\{a,b\}}$ is not a tree. It is an empty forest consisting of two empty trees with roots at vertices $a$ and $b$.

\end{example}
 
\subsection{Unweighted digraphs}

In the absence of a weight function, we can assume that  weights of all arcs are equal to one. Every spanning forest is also a minimal forest. If ${\cal F}^k\neq \emptyset$, then for $F\in{\cal F}^k$ $\Upsilon^F=|{\cal A}F|=N-k=\varphi^ k$.  If ${\cal F}^{k-1} = \emptyset$, then by definition $\varphi^{k-1}=\infty$, and the convexity inequalities take the form

\begin{equation*}
\phi^{k-1}-\phi^k > \phi^k- \phi^{k+1} = \phi^{k+1}-\phi^{k+2} = \ldots = \phi^{N-1}-\phi^N=1 \ ,  
\end{equation*}
whence from Property 5 we conclude that $\mathfrak{A}_{k+1}=\mathfrak{A}_{k+2}=\ldots = \mathfrak{A}_N$. Thus, all these algebras coincide with a Boolean whose atoms are single-element. The subgraph of any graph without loops, induced by a singleton subset, is the empty tree.   If the original digraph has at least one spanning tree ($k=1$), then the situation looks quite simple. The only non-Boolean algebra is the algebra $\mathfrak{A}_1$, and it is trivial $\mathfrak{A}_1=\{ {\cal N},\emptyset\}$. But, if ${\cal V}G={\cal N}$, then $G|_{\cal N}=G$. In particular, if $T\in{\cal F}^1$ (the forest $T$ is a spanning tree), then $T|_{\cal N}=T$ is still a tree. Thus, the general formulation of Theorem 5 can be narrowed to a single meaningful situation (which no longer needs proof).

\begin{theorem}
Let ${\cal F}^k\neq \emptyset$ and ${\cal F}^{k-1} = \emptyset$ for some $k>1$, and ${\cal E}\in\aleph_k$. Then $F|_{\cal E}$ is a tree for any $F\in{\cal F}^k$.
\end{theorem}

\subsection{Undirected graphs} 

The situation with undirected graphs looks most simple. Any undirected graph can be considered as directed by turning each edge into two opposing arcs. In this case, each $m$-vertex undirected tree is split into $m$ directed trees by enumerating all vertices designated as roots. This, in particular, means that in the undirected case all atoms of the algebra are labeled, and for them the hypothesis, according to Property 15, is true when (\ref{nonequal}) is satisfied. Note also that in this case, the algebra $\mathfrak{A}_k$ contains exactly $k$ atoms. Trees of any $k$-component minimal forest have  atoms of algebra $\mathfrak{A}_k$ themselves as sets of vertices. For undirected graphs, it is only necessary to check that, given the equal sign in the system of convexity inequalities, the restriction of the forest to an algebra atom is still a tree. That is, no special constructions are required and Theorem 1 is sufficient.

\section{The role of algebra atoms in Markov chains}

Diffusion processes under small random perturbations \cite{V2}, \cite{VF} lead to Markov chains with a finite number of states $N$, the distribution row vector of which satisfies the equation
\begin{equation*}
\dot{\vec x}=-\vec x{\bf L}^\varepsilon \ , \ \ L_{ii}^\varepsilon=-\sum_{j\neq i}L_{ij}^\varepsilon \ , 
\end{equation*}
where the transition densities (off-diagonal elements of the Laplace matrix) are exponentially small: $L_{ij}^\varepsilon\sim \exp (-V_{ij}/\varepsilon)$, $\varepsilon$ is a small parameter \cite{V}. Significantly different sub-processes \cite{V1},\cite{V2} are observed when executing (\ref{nonequal}) on exponentially large time scales $t=\tau \exp((\phi_{k-1}-\phi_k)/ \varepsilon)$. The dependence on a small parameter disappears in the limit, and in the slow time $\tau$ some of the details of the formed subprocess manage to be averaged out. Limit projectors are determined from the tree form of recording minors \cite{V3}-\cite{V5}, in which, due to unsignedness, only members corresponding to minimal forests survive. The minimal forests themselves are determined using the efficient algorithm \cite{V8}. Components belonging to one atom of the algebra $\mathfrak{A}_k$ are combined into one state. However, even with such a combination, the enlarged distribution vector $\vec x^k(\tau)$ does not have the property of stochastic continuity:

\begin{equation}
\lim_{\tau\searrow 0}\vec x^k(\tau)\neq \vec x^k (0) \ .
\end{equation}
There is equality only if the initial distribution $\vec x^k(0)$ is concentrated on the labeled atoms of the algebra $\mathfrak{A}_k$.  Theorems 5 and 6 allow us to naturally recalculate the enlarged transition probabilities from the initial $V_{ij}$ and determine the enlarged states themselves.



\centerline{Abstract}

\begin{center}{When a forest, narrowed to an atom of subset algebra, turns out to be a tree}
\end{center}

\centerline{Buslov V.A.}

\parbox[t]{12cm}
{\small It is proved that the restriction of the $k$ and $(k-1)$-component directed spanning forest of minimum weight to an atom of the subset algebra generated by the vertex sets of trees of $k$-component minimum spanning forests is a tree. For minimal spanning forests consisting of a smaller number of components, this property, generally speaking, does not exist. 
}
\vspace{0.5cm}

St. Petersburg State University, Faculty of Physics, Department of Computational Physics

198504 St. Petersburg, Old Peterhof, st. Ulyanovskaya, 3

Email: abvabv@bk.ru, v.buslov@spbu.ru
\end{document}